\documentclass[final,leqno,onefignum,onetabnum]{siamltex1213}
\usepackage[utf8]{inputenc}

\usepackage{fullpage,epsfig}
\usepackage{enumerate}
\usepackage{amsmath}
\usepackage{stmaryrd}
\usepackage{amssymb}
\usepackage{latexsym}
\usepackage{accents}
\usepackage{mathrsfs}
\usepackage{subfigure}
\usepackage{color}
\usepackage[colorinlistoftodos]{todonotes}

\numberwithin{equation}{section}

\newtheorem{example}{Example}[section]

\newtheorem{remark}[example]{Remark}
\newtheorem{construction}[example]{Construction}

 \def\norm#1#2{\Vert\,#1\,\Vert_{#2}}

\newcommand{\R}{{\mathbb R}}

\newcommand{\N}{{\mathbb N}}
\renewcommand{\det}{{\rm det}\, }

\newcommand{\dd}{\,\mathrm{d}}

\newcommand{\be}{\begin{eqnarray}}

\newcommand{\ee}{\end{eqnarray}}

\renewcommand{\d}{{\rm d}}

\newcommand{\md}{{\rm d}}

\renewcommand{\O}{\Omega}
\newcommand{\s}{\sigma}

\newcommand{\A}[1]{\langle#1\rangle}

\newcommand{\RK}{{\R_K^{2 \times 2}}}

\newcommand{\eps}{\varepsilon}

\newcommand{\GradQC}{\mathcal{GY}^{\mathcal{QC}}(\O;\R^{2\times 2})}

\newcommand{\abs}[1]{\left|#1\right|}
\newcommand{\Z}{\mathbb{Z}}
\newcommand{\unit}{{\bf e}}

\title{Gradient Young measures generated by quasiconformal maps in the plane}

\author{Barbora Bene\v{s}ov\'{a}$^{1,2}$
\and Malte Kampschulte$^{1}$
}

\begin{document}
\maketitle

\footnotetext[1]{Department of Mathematics I, RWTH Aachen University, D-52056 Aachen, Germany 
}
\footnotetext[2]{Institute of Mathematics, University of W\"{u}rzburg, Emil-Fischer-Stra\ss e 40, D-97074 W\"{u}rzburg, Germany}

\begin{abstract}

In this contribution, we completely and explicitly characterize Young measures generated by gradients of quasiconformal maps in the plane. By doing so, we generalize the results of Astala and Faraco \cite{AstalaFaraco} who provided a similar result for quasiregular maps and Bene\v{s}ov\'{a} and Kru\v{z}\'{i}k \cite{bbmk2013} who characterized Young measures generated by gradients of bi-Lipschitz maps. Our results are motivated by non-linear elasticity where injectivity of the functions in the generating sequence is essential in order to assure non-interpenetration of matter.
 
\end{abstract}



\section{Introduction}
In non-linear hyperelasticity one finds stable states by minimizing a prescribed \emph{stored energy} functional over the set of \emph{admissible deformations}. 

It is generally postulated \cite{ciarlet,silhavy} that an element of the set of admissible deformations should be  an \emph{orientation-preserving} and \emph{injective} map $y: \Omega \to y(\Omega)$ with a suitably integrable weak gradient; one usually assumes that $y\in W^{1,p}(\O;\R^n)$ with $1 < p \le +\infty$. Here and in the sequel $\Omega$ is a bounded Lipschitz domain, the reference configuration.

In the simplest case, we may consider stored energies $W \colon \R^{n\times n}\to\R$ that depend on the deformation only through its gradient; i.e. in order to find stable states one has to minimize the functional 
\be\label{motivation}
J(y):=\int_\O W(\nabla y(x))\,\md x\ ,
\ee
over the set of admissible deformations, possibly with some prescribed boundary data. It is natural to ask under which conditions on the set of admissible deformations and the stored energy we can find minima of $J(y)$. 

We shall take the approach of \emph{fixing a suitable set of deformations} and we will only be concerned with conditions on the stored energy; in a sense, this is complementary to the standard approach where the growth of the stored energy is fixed, which already determines a suitable set of deformations.

Thus, we set admissible deformations to be \emph{quasiconformal maps} \emph{in the plane}:
\begin{align}
\mathcal{QC}(\Omega; \R^2) = \Big\{& y \in W^{1,2}(\O;\R^2): \text{ $y$ is a homeomorphism and } \exists { K\geq 1} \text{ such that } \nonumber \\ &|\nabla y|^2 \leq K \det(\nabla y) \text{ a.e. in $\Omega$} \Big\}, \label{qc-def}
\end{align}
where { $|\cdot|$ is the spectral norm of a matrix.}

Before proceeding, let us motivate this choice from a mechanical point of view: By restricting our attention to homeomorphisms, we aim to model situations in which no failures like cracks or cavities occur. Further, we require in the spirit of \cite{bbmk2013} that the inverse $y^{-1}$ of a deformation is of the same ``quality'' as the deformation itself; here this is fulfilled since the inverse of a quasiconformal map is again quasiconformal (cf. e.g.\@ \cite[Theorem 3.1.2]{AstalaIwaniec}). The idea behind this restriction is that, in elasticity, a body returns to its original shape after the release of all loads. However, since the r\^{o}le of the reference and the deformed configuration is arbitrary, we would like to understand this ``returning'' as a new deformation, corresponding to inverse loads, that takes the deformed configuration to the reference one. Let us mention that the idea of also requiring integrability of the inverse of the deformation already appeared e.g.\@ in \cite{ball81,fonseca-gangbo,currents,tang} and very recently e.g.\@ in \cite{iwaniec1,iwaniec2,Henao,daneri-pratelli}; the use of quasiconformal maps in hyper-elasticity has been put forward in \cite{iwaniec3}.

Quasiconformal functions map infinitesimally small circles to infinitesimal ellipses of a uniformly bounded eccentricity. This means that in our modeling not even subsets of a vanishingly small measure can get deformed into a flat piece. Finally, quasiconformal mappings form, roughly speaking, the largest class of deformations that is invariant under composition with similarity transformations (i.e. shape preserving transformations in the domain and in the range).\footnote{If a family $\mathcal{F}$ of homeomorphisms of domains in $\mathbb{C}$ is normal and invariant under composition with similarity transformations (in the domain and in the range) then it consist of $K$-quasiconformal mappings with some $K \geq 1$ fixed \cite{Belinskii}.}

On the energy we include the additional requirement that it blows up if the volume of any infinitesimal part of the body shrinks to zero; i.e.
\begin{equation}\label{blowup}
W(A) \to +\infty \qquad \text{whenever} \qquad \det\, A \to 0_+.
\end{equation}

Since it is convenient to prove existence of minimizers by the \emph{direct method}, we \emph{characterize the set of energies $W$ satisfying \eqref{blowup} }such that $J(y)$ from \eqref{motivation} is \emph{weakly lower semicontinuous} on the set $\mathcal{QC}(\Omega; \mathbb{R}^2)$ (with respect to the weak convergence specified in Definition \ref{WeakConvQC}). It is expected that this will be connected to some notion of (quasi)convexity. 

Recall that we call $W$ quasiconvex \cite{morrey} if for all $A\in\R^{2\times 2}$ and all $\varphi\in W^{1,\infty}(\O;\R^2)$ such that $\varphi(x) = Ax $ on $\partial \Omega$  it holds that 
\be\label{quasiconvexity} 
|\O|W(A)\le\int_\O W(\nabla\varphi(x))\,\md x\ .\ee 
It is well known (cf., e.g., \cite{dacorogna}) that if a quasiconvex $W$ additionally satisfies the coercivity/growth condition
\begin{equation}
c(|A|^2-1) \leq W(A) \leq c(|A|^2+1),
\label{2growth}
\end{equation}
it is weakly lower semicontinuous on $W^{1,2}(\Omega; \mathbb{R}^2)$ and so in particular on the set $\mathcal{QC}(\Omega; \mathbb{R}^2)$. 

But \eqref{2growth} necessarily implies that $W$ is locally finite which is incompatible with \eqref{blowup}. Indeed, this was formulated as Problem 1 by J.M. Ball in \cite{ballOPEN} in the following way: \emph{``Prove the existence of energy minimizers for elastostatics for quasiconvex stored-energy functions satisfying \eqref{blowup}.''} Moreover, it has been recently shown in \cite{Filip2} that $W^{1,p}$-quasiconvexity with $p$ less than the dimension is even incompatible with \eqref{blowup} at all, so it seems natural to rather consider a \emph{natural generalization of quasiconvexity} tailored to quasiconformal functions. 

To this end, we introduce the concept of \emph{quasiconformally quasiconvex} functions (cf.\@ Def.\@ \ref{def:QC-quasiconvexity} below) that satisfy \eqref{quasiconvexity} \emph{only}
for all $\varphi \in \mathcal{QC}(\Omega;\R^2)$, $\varphi(x) = Ax$ on $\partial \Omega$ with $A\in\R^{2\times 2}$ having positive determinant. Clearly, this is a weaker notion than the usual $W^{1,2}$ - quasiconvexity in the sense of \cite{BallMurat} as well as quasiconvexity in the sense of \cite{morrey}. In particular, notice that now, since $W$ does not even need to be defined for matrices with negative determinant, we may as well set it to $+\infty$ on there.

In this contribution, we show that for stored energies satisfying the { following  growth condition perfectly fitted to the quasiconformal setting (cf. Remark \ref{rem:exponents}) and \emph{not contradicting} \eqref{blowup}  
\begin{equation}
0 \leq W(A) \leq c(|A|^{p} + |\det(A)|^{-q}  + 1), 
\end{equation}}
 with suitable $p=p(K)>2$ and $q=q(K)>0$, $J(y)$ is weakly lower semicontinuous  along $K$-quasiconformal sequences in the set $\mathcal{QC}(\Omega; \mathbb{R}^2)$ (in the sense of  Definition \ref{WeakConvQC}) \emph{if} it is quasiconformally quasiconvex. 

This claim follows from our main result, namely the complete and explicit characterization of gradient Young measures generated by sequences in  $\mathcal{QC}(\Omega; \mathbb{R}^2)$ (cf. Section \ref{results}). Young measures present a very convenient tool for studying weak lower semicontinuity when extending the notion of solutions from Sobolev mappings to parameterized measures \cite{ball3,fonseca-leoni,pedregal,r,tartar,tartar1,y}. The idea is to describe the limit behavior of $\{J(y_k)\}_{k\in\N}$ along a minimizing sequence $\{y_k\}_{k\in\N}$. Moreover, Young measures form one of the main relaxation techniques for non-(quasi)convex functionals as appearing when modeling solid-to-solid phase transitions \cite{ball-james1,mueller}. 

Generally speaking, the main difficulty in characterizing Young measures generated by a class of invertible mappings is that such a class is \emph{non-convex}. Thus, many of the classical techniques used in the study of Sobolev functions, like smoothing by a mollifier kernel, fail. In fact, as far as smoothing under the invertibility constraint is concerned, results in the plane were obtained only very recently in {\cite{iwaniec1,daneri-pratelli,mora1,mora2,mora3}} and are based on completely different ideas than integral kernels. From the point of view of the problem considered here, it is crucial to design a \emph{cut-off technique compatible with the invertibility constraint} that will allow us to modify a given function on a set of vanishingly small measure in such a way that it takes some prescribed form on the boundary. Indeed, the necessity of such a technique can be expected already from the very definition of (quasiconformal) quasiconvexity which requires us to verify \eqref{quasiconvexity} for maps with fixed boundary values and, in fact, cut-off techniques are exploited in all the standard proofs of characterizations of gradient Young measures \cite{k-p1,k-p2,pedregal} or weak lower semicontinuity of quasiconvex functionals \cite{dacorogna}. 

To the best of our knowledge, gradient Young measures generated by invertible maps were so far characterized only in \cite{bbmk2013}, where an explicit characterization of Young measures generated by bi-Lipschitz maps in the plane was given. Moreover, the particular case of \emph{homogeneous} gradient Young measures generated by quasiconformal maps has been treated in \cite{AstalaFaraco}, but in the non-homogeneous case the global invertibility was given up and only Young measures generated by quasiregular maps were characterized. The present contribution generalizes the above two works in the sense that we adapt the cut-off technique from \cite{bbmk2013} also to the quasiconformal case and so characterize non-homogeneous gradient Young measures generated by invertible quasiregular maps. 

We follow the main strategy of \cite{bbmk2013} to construct the cut-off, i.e.\@  we modify the given sequence on a set of gradually vanishing measure near the boundary first on a one dimensional grid and then rely on extension theorems for quasiconformal maps going back to \cite{BeurlingAhlfors}. Nevertheless, contrary to the bi-Lipschitz case the modification on the grid is much more involved and cannot be done by what is essentially an affine interpolation as in the Lipschitz case \cite{bbmk2013} (cf. Section \ref{sect-cutOff}). Therefore, it seems feasible that the presented constructions on the grid will carry over also to more general settings of, e.g., bi-Sobolev maps. Nevertheless, the extension from the grid or, in other words the full characterization of traces, is completely open in these more general situations.  

Let us note that even if one poses only point-wise constraints on the determinant of the deformation (e.g.\@ by requiring $\det(\cdot) > 0$), similar difficulties to the ones described above arise. However, in some less rigid situations, one may rely on convex integration to construct cut-offs. Such an approach has been taken in \cite{bbmkgpYm} as well as in \cite{krw,Filip2}, where Young measures generated by orientation preserving maps in $W^{1,p}$ for $1<p<n$ were characterized.

The plan of the paper is as follows: In the rest of the introduction, we give some background on Young measures and explain some basic concepts that we shall use throughout the contribution. Then we state the main results in 
Section~\ref{results} and recall some facts about quasiconformal maps in Section \ref{Start}. Proofs of the main theorems are postponed to Section~\ref{Proofs} while our cut-off technique is presented in Section \ref{sect-cutOff}.

\subsection{Background on gradient Young measures}

We understand Young measures as ``generalized functions'' that capture the asymptotic behavior of a non-linear functional along an oscillating sequence $\{Y_k\}_{k \in \mathbb{N}}$. Namely, suppose that $\{Y_k\}_{k \in \mathbb{N}}$ is bounded in $L^{2}(\Omega; \mathbb{R}^{2\times 2})$, then a classical result \cite{tartar,warga,ball3,MKTR}, the \emph{fundamental Young measure theorem}, states that there exists a sub-sequence (not relabeled) of $\{Y_k\}_{k \in \mathbb{N}}$ and a family of probability measures $\nu=\{\nu_x\}_{x\in\O}$ satisfying
\be\label{jedna2}
\lim_{k\to\infty} \int_\Omega \xi(x) (v \circ Y_k) (x) \dd x =\int_\Omega \int_{\R^{2\times 2}}\xi(x) v(s)\nu_x(\d s) \dd x\ \
\ee
for all $\xi \in L^\infty(\Omega)$ and all $v \in C(\R^{2 \times 2})$ such that $\{v \circ Y_k\}_{k \in \mathbb{N}}$ is weakly convergent in $L^1(\Omega)$. We say that $\{Y_k\}_{k \in \mathbb{N}}$ \emph{generates} the Young measure $\nu=\{\nu_x\}_{x\in\O}$. It is further known that $\nu_x  \in L^\infty_{\rm
w*}(\O;\mathcal{M}(\R^{2\times 2}))\cong L^1(\O;C_0(\R^{2\times 2}))^*$ where $\mathcal{M}(\R^{2\times 2})$ is the set of Radon measures, $C_0(\R^{2\times 2})$ stands for
the space of all continuous functions $\R^{2\times 2}\to\R$ vanishing at infinity and space $L^\infty_{\rm w*}(\O)$ corresponds to weakly* measurable uniformly bounded functions. Recall that weakly* measurable means that,
for any $v\in C_0(\R^{n\times n})$, the mapping
$\O\to\R:x\mapsto\A{\nu_x,v}=\int_{\R^{n\times n}} v(s)\nu_x(\d s)$ is
measurable in the usual sense.

An important subset of Young measures are those generated by \emph{gradients} of $\{y_k\}_{k\in\N}\subset W^{1,2}(\O;\R^2)$, i.e.,  $Y_k:=\nabla y_k$ in \eqref{jedna2}. Let us denote this set ${\cal GY}^2(\O;\R^{2\times 2})$. An explicit characterization of this set is due to Kinderlehrer and Pedregal \cite{k-p1,k-p2}:
\begin{theorem}[adapted from \cite{k-p2}]
\label{k-p-thm}
A family of Young measures $\{\nu_x\}_{x \in \Omega}$ is in ${\cal GY}^2(\O;\R^{2\times 2})$ if and only if
\begin{enumerate}
  \item  there exists $z \in W^{1,2}(\Omega; \R^2)$ such that $\nabla z(x) = \int_{\R^{2\times 2}} A \nu_x(\d A)$  for a.e.\@ $x \in \Omega$,
  \item $\psi(\nabla z(x)) \leq \int_{\R^{2\times 2}}\psi(A) \nu_x(\d A) $ for a.e. $x \in \Omega$ and for all $\psi$ quasiconvex, continuous and such that $|\psi(A)| \leq c(1+|A|^2)$,
  \item $\int_\Omega \int_{\R^{2\times 2}} |A|^2  \nu_x(\d A) \dd x  < \infty$.
\end{enumerate}
\end{theorem}

\subsection{Basic notation}

We define the set of $K$-quasiconformal matrices 
\begin{align}
\label{R_K} \RK&:=\{A\in \mathbb{R}^{2 \times 2}:\ |A|^2 \leq K {\rm det} (A) \}\ ,
\end{align}
for $1\leq K<\infty$. Note that $\RK$ represents (except for the zero-matrix) the possible values of gradients of affine $K$-quasiconformal mappings. Further, we denote the set of matrices in $\R^{2 \times 2}$ with positive determinant as $\R^{2 \times 2}_+$.   

Next, let us introduce a suitable notion of weak convergence on $\mathcal{QC}(\Omega;\R^2)$: 
\begin{definition}
\label{WeakConvQC} \mbox{}
We say that a sequence $\{u_k\}_{k \in \mathbb{N}} \subset W^{1,2}(\Omega;\R^2)$ of quasiconformal maps converges weakly to $u \in W^{1,2}(\Omega; \mathbb{R}^2)$ in $\mathcal{QC}(\Omega;\R^2)$ if $u_k \rightharpoonup u$ in $W^{1,2}(\Omega; \mathbb{R}^2)$, there exists a $K \geq 1$ such that the $u_k$ are all $K$-quasiconformal and $u(x)$ is non-constant. If the weak convergence is given only in $W^{1,2}_\mathrm{loc}(\Omega; \mathbb{R}^2)$, we will speak of local convergence in $\mathcal{QC}(\Omega; \R^2)$. 
\end{definition}

\begin{remark} \mbox{}
\label{rem-convQC}
Notice that the set $\mathcal{QC}(\Omega;\R^2)$ is closed under the weak convergence in $\mathcal{QC}(\Omega;\R^2)$. Indeed, it follows from Lemma \ref{lem:HighInt} (below) that if a sequence converges in $\mathcal{QC}(\Omega;\R^2)$, it also converges locally uniformly and, thus, the weak limit is either quasiconformal or constant, with the latter possibility however being excluded per definition. { For the fact, that the limit of a locally uniformly converging sequence of K-quasiconformal maps can be only K-quasiconformal or constant we refer to \cite[Thm.2.2]{lehto} or (in a slightly different formulation) to \cite[Thm. 3.1.3]{AstalaIwaniec}.}
\end{remark}
 
\section{Main results}
\label{results}

We shall denote
\begin{align*}
\GradQC = \big\{ \nu &\in \mathcal{GY}^{2}(\Omega; \R^{2 \times 2}) \text { generated by maps weakly converging in $\mathcal{QC}(\Omega; \mathbb{R}^2)$} \big\}.
\end{align*}

The main results of our paper are formulated in Theorems \ref{THM1}, \ref{THM2} and \ref{THM3}, below. In Theorem \ref{THM1}, we give a complete and explicit characterization of $\GradQC$. Further, a relation between Young measures in $\GradQC$ and quasiconformally quasiconvex functions is observed in Theorem \ref{THM2}. From this, we readily deduce an equivalent characterization of weak lower semicontinuity of functionals on $\mathcal{QC}(\Omega; \R^2)$ in Theorem \ref{THM3}.

We start with the characterization of measures in $\GradQC$:
\begin{theorem}\label{THM1}
Let $\Omega \subset \R^2$ be a bounded Lipschitz domain. Let $\nu\in \mathcal{GY}^2(\O;\R^{2\times 2})$. Then $\nu \in \GradQC$ if and only if the following conditions hold:
 \begin{gather}
 \exists K \geq 1 \text{ such that }\, \mathrm{supp} \, \nu_x\subset \RK \text{ for a.a.\@ $x\in \Omega$ and} \label{supp} \\
    \text{ there exists a $K$-quasiconformal $y \in {\mathcal{QC}}(\Omega; \R^2)$ such that }\nabla y(x)=\int_{\R^{2\times 2}} A \d \nu_x(A) \text{.}  \label{firstmoment0} 
\end{gather}   
\end{theorem}
{
\begin{remark}
Let us remark, that the above theorem could be slightly weakened by requiring that $\nu\in \mathcal{GY}^p(\O;\R^{2\times 2})$ with $p > \frac{2K}{K+1}$ and $K$ corresponding to \eqref{supp} and \eqref{firstmoment0}.  This is due to the self-improving property of quasiconformal maps used already in \cite{AstalaFaraco} to prove a version of Theorem \ref{homYM} under the above assumptions. On the other hand, even if $p$ was bigger than 2, the generating sequence that we construct will still be in $\mathcal{QC}(\Omega; \R^2)$ and thus its (in general) best local regularity is given by Lemma \ref{lem:HighInt}.   

In this paper we will stick to the exponent $p=2$ since it is the canonical choice for quasiconformal maps; cf. the definitions in \cite{AstalaIwaniec}.  Moreover, even for maps with only an integrable distortion, which form a natural generalization on quasiconformal maps \cite{Hencl} and could form a natural class of deformations in elasticity, $p$ equal to dimension is the most natural choice. In such a case, the inverse of the deformation will be in $W^{1,2}$ again and thus a symmetry between the deformation and its reversal is recorded.
\end{remark}
}

{
\begin{remark}
It will follow from the proof of Theorem \ref{THM1} that a generating sequence for $\nu\in \mathcal{GY}^2(\O;\R^{2\times 2})$ satisfying \eqref{supp} and \eqref{firstmoment0} can be chosen in such a way that it coincides with $y$ on the boundary of $\Omega$ (cf. also Proposition \ref{prop-cutOff}).

However, the construction used for the proof of Theorem \ref{THM1} does not guarantee that the generating sequence will be also $K$-quasiconformal; in fact, it may just be $\kappa(K)$-quasiconformal with $\kappa(K)$ only depending on $K$, but possibly larger than $K$. This is a drawback that {requires us to use regularizations in order }to prove existence of minimizers for quasiconformally quasiconvex functions (cf. also Remark \ref{ex-minimizers} below). Nevertheless, let us note that this drawback is not introduced by the fact that we study injective generating sequences but it appears already in the characterization of $L^\infty$-gradient Young measures in \cite{k-p1,pedregal}. 
\end{remark}
}

Notice that  similarly to \cite{AstalaFaraco}, we constrained the support of the given Young measure to a suitable set of quasiconformal matrices in Theorem \ref{THM1}. Also, since the set $\mathcal{QC}(\Omega; \mathbb{R}^{2})$ is closed under the weak convergence in $\mathcal{QC}(\Omega; \mathbb{R}^{2})$ the condition on the first moment \eqref{firstmoment0} is natural.

In view of previous results, it might seem surprising that no adaptation of the Jensen inequality is needed. Nevertheless, we will show in Theorem \ref{THM2} that measures in $\mathcal{GY}^\mathcal{QC}(\O;\R^{2\times 2})$ satisfy, in fact, a more restrictive version of the Jensen inequality perfectly fitted to quasiconformal maps. To this end, let us introduce the following generalized notion of quasiconvexity:

\begin{definition}
\label{def:QC-quasiconvexity}
Suppose $v:\R^{2\times 2}_+\to\R\cup\{+\infty\}$ is bounded from below and  Borel measurable. We say that $v$ is { $K$}-quasiconformally quasiconvex on $\R^{2 \times 2}_+$ if 
\begin{equation}
|\Omega| v(A) \leq \int_\Omega v(\nabla \varphi(x))\,\md x\
\label{ineq-QC:convex}
\end{equation}
for all $A \in \R^{2 \times 2}_+$ and all { $K$-quasiconformal} $\varphi \in \mathcal{QC}(\O;\R^2)$ such that $\varphi(x) = Ax$ on $\partial \Omega$. { Further, we call a function quasiconformally quasiconvex if it is $K$-quasiconformally quasiconvex for all $K \geq 1.$}
\end{definition}

\begin{remark}[Relation to other notions of quasiconvexity]
Notice that quasiconformal quasiconvexity is a weaker condition than $W^{1,2}$-quasiconvexity since all quasiconformal functions are by definition in $W^{1,2}(\Omega; \R^2)$, while the opposite is by far not true. In fact, a quasiconformally quasiconvex function can be completely arbitrary on the set of matrices with negative determinant, while in general this is not true for $W^{1,2}$-quasiconvex functions. 

A (in general) even weaker condition than the one from Definition \ref{def:QC-quasiconvexity}, so-called \emph{bi-quasiconvexity}, has been introduced in \cite{bbmk2013}; this notion is based on verifying the Jensen inequality \eqref{ineq-QC:convex} just for \emph{bi-Lipschitz} maps. In order to prove that these two notions are equivalent, one would need to assure density of bi-Lipschitz maps in quasiconformal ones in a suitable strong convergence respecting the growth of the function $v: \R^{2\times 2}\to\R\cup\{+\infty\}$ on the set of matrices with positive determinant. For example, one would seek a result showing that for every $K$-quasiconformal function there exists a sequence of $K$-quasiconformal bi-Lipschitz maps that coincide with the original function on $\partial \Omega$ and approximate the given function strongly in the $W^{1,2}$-norm and their inverse Jacobians converge strongly to the inverse Jacobian of the original function in the $L^1$-norm. However, density results on homeomorphisms started to appear only recently in literature \cite{iwaniec1,daneri-pratelli,mora1,mora2,mora3} and a result of the type mentioned above is currently not available to the authors' knowledge.    
\end{remark}

With this definition we have the following theorems:
\begin{theorem}
\label{THM2}
Any Young measure $\nu \in \GradQC$, { that can be generated by sequence of $K$-quasiconformal maps,} satisfies the following inequality 
\begin{equation}
     v(\nabla u(x))\le\int_{\R^{2\times 2}} v(A) \d \nu_x(A)\ \label{qc0}
 \end{equation}
for all { $\kappa(K)$}-quasiconformally quasiconvex $v$ in $\mathcal{E}{(K)}$ and a.a. $x \in \Omega$;
where
\begin{equation}
\mathcal{E}{(K)}:=\big \{ v:\R^{2\times 2}\to\R\cup\{+\infty\};\ v\in C(\R^{2 \times 2}_+)\ ,\ 0 \leq v(A)  \leq C\big(1+|A|^{ p} + |\mathrm{det}(A)|^{-q}\big) \ \big\}\,
\label{Orho}
\end{equation}
{ for any $p < \frac{2\kappa(K)}{\kappa(K)-1}$ and $q< \frac{1}{\kappa(K)-1}$ where $\kappa(K)$ only depends on $K$ and is found in Proposition \ref{prop-cutOff}.} 
\end{theorem}

{This theorem readily yields that quasiconformal quasiconvexity implies lower semicontinuity, as shown in Theorem \ref{THM3}. For proving the reverse, i.e. that quasiconformal quasiconvexity is also necessary for lower semicontinuity, we however need to rely on the characterization of $\GradQC$ in Theorem \ref{THM1}.} 

\begin{theorem}
\label{THM3}
Let $v \in \mathcal{E}{(K)}$ and let $\{y_k\}_{k \in \N} \subset \mathcal{QC}(\Omega;\R^2)$ be a sequence of { $K$-quasiconformal} maps that converge weakly in $\mathcal{QC}(\Omega;\R^2)$. Then $y\mapsto I(y):=\int_\O v(\nabla y(x))\,\d x$  is sequentially weakly lower semicontinuous { along any such sequence} if $v$ is { $\kappa(K)$}-quasiconformally quasiconvex. { On the other hand, if $I(y)$ is weakly lower semicontinuous along any sequence of K-quasiconformal maps converging weakly in $\mathcal{QC}(\Omega;\R^2)$ then it is $K$-quasiconformally quasiconvex.}
\end{theorem}
{

An easy corollary of the above theorem that any $v \in \bigcap_{K \geq 1} \mathcal{E}(K)$ is weakly lower semicontinuous on $\mathcal{QC}(\Omega;\R^2)$ \emph{if and only if} it is quasiconformally quasiconvex.

\begin{remark}[Exponents in $\mathcal{E}{(K)}$]
\label{rem:exponents}
Let us remark that the exponents $p > 2$, $q > 0$ in $\mathcal{E}(K)$ are perfectly fitted to the $K$-quasiconformal setting. In fact, due to Lemma \ref{lem:HighInt}, we know that every $K$-quasiconformal function is locally in $L^{p}(\Omega; \mathbb{R}^2)$ with $p < \frac{2K}{K-1}$ and this bound is optimal \cite[Chapter 13]{AstalaIwaniec}. The higher intergrability then yields (cf. for example Lemma \ref{lem:non-conc} below) that the inverse Jacobian is integrable with the power $1/(K-1)$ which is again the optimal exponent as the function $x|x|^{K-1}$ shows (see e.g. \cite{koskela} in a more general setting). 
\end{remark}
}

{
\begin{remark}[Further characterization of $\GradQC$]
Notice that a characterization of $\GradQC$ in the spirit of Theorem \ref{k-p-thm} is an easy corollary of Theorems \ref{THM1} and \ref{THM2}: A family of Young measures $\{\nu_x\}_{x \in \Omega}$ is in $\GradQC$ if and only if
\begin{enumerate}
  \item  $\exists K \geq 1 \text{ such that }\, \mathrm{supp} \, \nu_x\subset \RK \text{ for a.a.\@ $x\in \Omega$}$
\item there exists a $K$-quasiconformal $y \in \mathcal{QC}(\Omega; \R^2)$ such that $\nabla y(x)=\int_{\R^{2\times 2}} A \d \nu_x(A)$.
  \item $\psi(\nabla z(x)) \leq \int_{\R^{2\times 2}}\psi(A) \nu_x(\d A) $ for a.e. $x \in \Omega$ and for all $\psi \in \bigcap_{K \geq 1} \mathcal{E}(K)$ quasiconformally quasiconvex, 
  \item $\int_\Omega \int_{\R^{2\times 2}} |A|^2  \nu_x(\d A) \dd x  < \infty$.
\end{enumerate}
In fact, since all quasiconformally quasiconvex functions are also quasiconvex, we deduce that any $\nu$ satisfying the four conditions above is in $\mathcal{GY}^2(\Omega; \mathbb{R}^{2 \times 2})$ and so, according to Theorem \ref{THM1}, also in $\GradQC$. On the other hand, if $\nu \in \GradQC$, we know that it is, in particular, an element of $\mathcal{GY}^2(\Omega; \mathbb{R}^{2 \times 2})$ so that Theorems \ref{k-p-thm} and \ref{THM1} yield the items 1,2 and 4 above while the last item is a consequence of Theorem \ref{THM2}.
\end{remark}
} 
\begin{remark}[Existence of minimizers]
\label{ex-minimizers}
Clearly, a weak lower semicontinuity result is essential in order to prove existence of minimizers for energies with a density from $\mathcal{E}(K)$. Nevertheless, if the set of functions on which the minimization  is performed features some $L^\infty$-type constraint (such as the constraint on the distortion in our case), a coercivity-related issue prevents a simple application of the direct method. In fact, $L^\infty$-type constraints can be enforced by letting the stored energy density be finite only on a suitable subset of $\R^{2 \times 2}$; yet, this subset is usually left when employing cut-off methods---this happens even in the standard cases \cite{dacorogna}. 

{A possible remedy is to add a (presumably small) penalization term that enforces the $L^\infty$-constraint in question. We present a corresponding result in Theorem \ref{minimizers} where we penalize large distortion. Such a penalization term has to be of a non-local character and thus cannot be understood via an energy density; nevertheless, one can argue that it effectively affects the energy only if, at least locally, the distortion has become large which might indicate that the underlying atomic lattice is faulted and the elastic approximation is not appropriate anymore.}

The usual remedy for proving existence of minimizers or, connected to this, relaxation results is to work with $L^p$-type (with $p$ finite) constraints only, as in \cite{Conti,hm} where some particular results for relaxation under determinant constraints were proved. In our setting, this would mean to work with homeomorphisms whose distortion is just integrable. Indeed, such a class of functions has been in the focus of recent research \cite{Hencl} but what is still missing to adapt our technique to such a class is an equivalent to the extension property from Lemma \ref{lemma-extension}.
\end{remark}

{
We now state a theorem about the existence of minimizers:

\begin{theorem}
\label{minimizers}
Suppose that $\Omega \subset \mathbb{R}^2$ is a Lipschitz domain with $\partial \Omega$ being a quasicircle. Let $v \in \bigcap_{K \geq 1} \mathcal{E}(K)$ be quasiconformally quasiconvex such that $v(A) = +\infty$ if $\mathrm{det(A) \leq 0}$ and let it satisfy the coercivity condition
$
v(A) \geq c( -1+|A|^2)
$.
Further, let $\varepsilon > 0$ and set
$$
J(y):= \int_\Omega v(\nabla y) \dd x + \varepsilon \left\|\frac{|\nabla y|^2}{\mathrm{det}(\nabla y)} \right\|_{L^\infty(\Omega)}.
$$
Finally, let $\eta: \partial \Omega \to \mathbb{R}^2$ be quasi-symmetric with $\eta(\partial \Omega)$ a simple closed curve and let $\Gamma \subset \partial \Omega$ be of positive one-dimensional measure. 

Then $J$ possesses a minimizer on the set of maps in $W^{1,2}(\Omega;\R^2)$ that satisfy the Ciarlet-Ne\v{c}as condition  \eqref{Ciarlet-Necas} and that coincide with $\eta$ on $\Gamma$ in the sense of trace. Moreover, the minimizer is found in $\mathcal{QC}(\Omega; \R^2)$.
\end{theorem}

We refer here to Lemma \ref{lemma-locQuasi} for a definition of quasisymmetry and to Lemma \ref{lemma-astala-quasicircle} for a characterization of quasicircles.
}

\bigskip

\section{Some auxiliary results on quasiconformal maps}
\label{Start}

Quasiconformal maps have been studied intensively for several decades now; cf.\@ e.g.\@ the monographs \cite{Ahlfors,AstalaIwaniec} for further details. For the convenience of the reader, let us recall some of their properties that shall be of importance in this work. Note that the classical theory on quasiconformal maps, as presented in e.g.\@ \cite{AstalaIwaniec}, does not treat the class $\mathcal{QC}(\Omega; \R^2)$ but rather the following set
\begin{align}
\mathcal{QC}_\mathrm{loc}(\Omega; \R^2) = \Big\{& y \in W^{1,2}_\mathrm{loc}(\O;\R^2): \text{ $y$ is a homeomorphism and } \exists { K\geq 1} \text{ such that } \nonumber \\ &|\nabla y|^2 \leq K \det(\nabla y) \text{ a.e. in $\Omega$} \Big\}.
\end{align}
which, in this work, shall be referred to as locally quasiconformal maps. Let us also mention that some of the results given below can be extended even to maps the distortion of which is only integrable \cite{Hencl}.

\begin{lemma}[Inverse and composition, adapted from Theorem 3.1.2 in \cite{AstalaIwaniec}]
Let $u: \Omega \to \mathbb{R}^2$ be a K-quasiconformal map. Then its inverse is K-quasiconformal, too. Moreover, the composition of a $K_1$-quasiconformal map and a $K_2$-quasiconformal map is $K_1K_2$-quasiconformal.
\end{lemma}

\begin{lemma}[Higher integrability, adapted from Theorem 13.2.3 in \cite{AstalaIwaniec}\footnotemark] \footnotetext{In fact the second part of the Theorem is obtained by combining Lemma 3.6.1 and Lemma 2.10.9 with the above mentioned theorem.}
\label{lem:HighInt}
Let $u: \Omega \to \mathbb{R}^2$ be a K-quasiconformal map. Then $u \in W^{1,p}_\mathrm{loc}(\Omega; \R^2)$ for all $p < \frac{2K}{K-1}$. More specifically, take $x_0 \in \Omega$ and $r > 0$ such that the ball $B_{2r}(x_0) \subset \Omega$. Then there exists a constant $c$ independent of $r, x_0$ and $K$ such that
$$
\int_{B_r(x_0)} |\nabla u|^{p} \dd x \leq c.
$$

In particular, a sequence $\{u_k\}_{k \in \mathbb{N}}$ that converges weakly to $u$ in $\mathcal{QC}(\Omega; \R^2)$ converges also locally uniformly to $u$.
\end{lemma}

It is shown in \cite{Gehring} that gradients of quasiconformal maps have locally a better integrability than just $L^2$, the precise bound is derived in \cite{Astala}. 

\begin{lemma}[Local quasisymmetry, adapted from Theorem 3.6.2 in \cite{AstalaIwaniec}] \label{lemma-locQuasi}
Let $u: \Omega \to \mathbb{R}^2$ be a K-quasiconformal map; further, take $x_0 \in \Omega$ and $r > 0$ such that the ball $B_{2r}(x_0) \subset \Omega$. Then $u_{|B_r(x_0)}$ is \emph{quasisymmetric}, i.e. it is a homeomorphism and there exists an increasing function $\eta: \mathbb{R}_+ \to \mathbb{R}_+$ such that for any triple $x,y,z \in B_r(x_0)$ the following is satisfied
$$
\frac{|u(x)-u(y)|}{|u(x)-u(z)|} \leq \eta \left(\frac{|x-y|}{|x-z|} \right)
$$
and the function $\eta$ depends only on $K$ but not on $x_0, r$.
\end{lemma}

The following is a direct consequence of the famous extension of Beurling and Ahlfors \cite{BeurlingAhlfors}: 
\begin{lemma}[Extension property] \label{lemma-extension}
Suppose that $D$ is a square in $\R^2$ and $u: \partial D \to u(\partial D)$ is $\eta$-quasisymmetric and that $u(\partial D)$ is a simple closed curve. Then there exists a $K$-quasiconformal map $\tilde{u}: D \to \mathrm{Int} (u(\partial D))$ which coincides with $u$ on $\partial \Omega$. Moreover, $K$ depends only on $\eta$.
\end{lemma}
{
Notice that due to the Jordan mapping theorem $u(\partial D))$ has a well defined interior (which we denote my $\mathrm{Int}(u(\partial D))$ as well as an exterior $\mathrm{Ext}(u(\partial D))$.

\begin{proof}(of Lemma \ref{lemma-extension})
First, owing to the Riemann mapping theorem, we find the conformal maps $\phi_1: D \to \mathbb{H}$ and $\phi_2: \mathrm{Int} (u(\partial D)) \to \mathbb{H}$ with $\mathbb{H}$ being the upper half-plane. Since $D$ as well as $u(\partial D)$ are quasicircles, it follows by the reflection principle of Ahlfors \cite{Ahlfors-refl} (cf. also \cite[Lemma 2]{reed}) that $\phi_1$, $\phi_2$ can be extended quasiconformally to mappings of the whole plane and therefore, as quasiconformal mappings of the whole plane are quasi-symmetric \cite[Thm. 3.5.3]{AstalaIwaniec}, extend also to quasisymmetric mappings $\bar{\phi}_1: \partial D \to \partial \mathbb{H}$ and $\bar{\phi}_2: u(\partial D) \to \partial \mathbb{H}$, the quasi-symmetry modulus of which depends only on $\eta$ (in fact for $\bar{\phi}_1$ it is even independent of $\eta$ and completely determined by the fact that $D$ is a square). Since inversions and compositions of quasi-symmetric maps are also quasi-symmetric, $\bar{\phi}_2 \circ  u \circ\bar{\phi}_1^{-1}$ is a $\tilde{\eta}$-quasi-symmetric mapping from the $x$-axis to the $x$-axis with $\tilde{\eta}$ depending only on $\eta$. 

Thus, relying on the Beurling-Ahlfors extension \cite{BeurlingAhlfors}, we may extend this map to a locally $K$-quasiconformal mapping $\mathfrak{b}: \mathbb{H} \to \mathbb{H}$ with $K$ depending only on $\tilde{\eta}$ and hence $\eta$. By setting $\tilde{u} =   \phi_2^{-1}\circ \mathfrak{b} \circ \phi_1$ we obtain $\tilde{u} \in \mathcal{QC}_\mathrm{loc}(D;\mathrm{Int} (u(\partial D)))$, which however is actually in $W^{1,2}(D; \mathbb{R}^2)$ due to \eqref{Ciarlet-Necas}, since $\mathrm{Int} \,u(\partial D)$ is bounded.
\end{proof}

In fact, the extension property still holds if we replace the square $D$ in Lemma \ref{lemma-extension} by a \emph{quasicircle}, that is the \emph{image of a circle under a quasisymmetric homeomorphism}. We will use the fact that quasicircles can be characterized by a kind of reverse triangle inequality: 

\begin{lemma}[Characterization of quasicircles, adapted from Theorem 13.3.1 in \cite{AstalaIwaniec}] \label{lemma-astala-quasicircle}
Let $\mathcal{C}$ be a closed curve in $\mathbb{R}^2$. Then $\mathcal{C}$ is the image of a circle under an $\eta$ - quasisymmetric homeomorphism if and only if there exists a constant $c$ which depends only on $\eta$ such that for any two points $z_1$ and $z_2$ chosen on the given closed curve and $z_3$ lying on the shorter of the resulting arcs, we have 
\begin{equation}
|z_1-z_3| + |z_2-z_3| \le c |z_1-z_2|.
\label{Arc-condition}
\end{equation}
\end{lemma}

A natural generalization of quasiconformal maps are \emph{quasiregular} maps; i.e., those that are of bounded distortion but not necessarily homeomorphisms. Let us point out through the following lemma that one of the possibilities to assure that such maps are injective is by imposing the Ciarlet-Ne\v{c}as condition \cite{ciarlet-necas} that is well known in elasticity.

\begin{lemma}[Quasiregularity and Ciarlet-Ne\v{c}as condition]
\label{lemma:C-N}
A map $u \in W^{1,2}(\Omega; \mathbb{R}^2)$ is $K$-quasiconformal if and only if it is non-constant, K-quasiregular and satisfies the Ciarlet-Ne\v{c}as condition
\begin{equation}
\int_\Omega \mathrm{det}(\nabla u) \dd x \leq |u(\Omega)|.
\label{Ciarlet-Necas}
\end{equation}
\end{lemma}

\begin{proof}
Clearly, in order to be a homeomorphism $u$ cannot be constant. Moreover, since quasiregular maps are continuous (or more specifically have a continuous representative), open (that is map open sets to open sets) and discrete (the set of pre-images for any point does not accumulate) (cf. e.g.\@ \cite[Corollary 5.5.2]{AstalaIwaniec},\cite{Hencl}), we only have to prove that the additional condition \eqref{Ciarlet-Necas} guarantees (and is implied by) injectivity. 

The proof of this follows from the \emph{area formula}. Namely, as both quasiregular and quasiconformal maps satisfy the Lusin $N$-condition (i.e. map sets of zero measure to maps of zero measure) (cf.\@ e.g.\@ \cite{Hencl}), we have that 
$$
\int_\Omega \mathrm{det}(\nabla u) \dd x = \int_{\mathbb{R}^2} N(u,\Omega, y) \dd y = \int_{u(\Omega)} N(u,\Omega, y) \dd y
$$
where $N(u,\Omega, y)$ is defined as the number of pre-images of $y$ in $\Omega$. So the Ciarlet-Ne\v{c}as condition is satisfied if and only if $N(u,\Omega,y) = 1$ almost everywhere on $u(\Omega)$. Also we can immediately see that the reverse inequality to (\ref{Ciarlet-Necas}) always holds.

If $u$ is injective, then $N(u,\Omega,y)=1$ and \eqref{Ciarlet-Necas} is satisfied. For the converse, suppose by contradiction, that there is a non-injective quasiregular, non-constant $u$ satisfying \eqref{Ciarlet-Necas}. Then there has to exist a $y \in u(\Omega)$ that has at least to two pre-images $x_1$ and $x_2$. Now there exists an $\eps > 0$ such that $B_\eps(x_1) \cap B_\eps(x_2) = \emptyset$ and $B_\eps(x_j) \subset \Omega$ for $j=1,2$. On the other hand, for the images we have that $u(B_\eps(x_1)) \cap u(B_\eps(x_2)) \neq \emptyset$. In fact, $u(B_\eps(x_1)) \cap u(B_\eps(x_2))$ is of positive measure since both $u(B_\eps(x_1))$ and $u(B_\eps(x_2))$ are open. Therefore, there exists a set of positive measure where $N(u,\Omega, y)$ is at least two; a contradiction to \eqref{Ciarlet-Necas}.
\end{proof}

\begin{lemma}[Gluing of quasiconformal maps] 
\label{lem:Gluing}
Let $\{\Omega_i\}_{i \in \mathbb{N}}$ be mutually disjoint simply connected Lipschitz domains that almost cover $\Omega$, i.e.\@ $\Omega = \bigcup_{i \in \mathbb{N}} \Omega_i \cup \mathcal{N}$ with $|\mathcal{N}|=0$. Further, let $u_i: \Omega_i \to \mathbb{R}^2$ be K-quasiconformal maps satisfying $u_i(x) = u(x)$ on $\partial \Omega_i$ with $u: \Omega \to \mathbb{R}^2$ also K-quasiconformal. Then the ``glued map''
$$
\tilde{u}(x) = \begin{cases}
u_i(x) & \text{if $x \in \Omega_i$,}\\
u(x) & \text{else,}
\end{cases}
$$
is K-quasiconformal as well.
\end{lemma}

In order to prove the ``gluing lemma'' we will exploit the characterization by the Ciarlet-Ne\v{c}as condition from Lemma \ref{lemma:C-N}. Alternatively, it is known that an open and discrete mapping equal to a homeomorphism near the boundary is already injective \cite{Hencl}, which would allow us to show the lemma, too.

\begin{proof}
Clearly $\tilde{u}$ is non-constant and $K$-quasiregular. To see that it is also $K$-quasiconformal, we verify \eqref{Ciarlet-Necas}. But since $\mathrm{det}(\cdot)$ is a null-Lagrangian or, alternatively, by applying the area formula, we have that 
$$
\int_{\Omega_i} \det(\nabla \tilde{u}) \dd x = \int_{\Omega_i} \det(\nabla u) \dd x,
$$
because from construction $\tilde{u}(x) = u(x)$ on $\partial \Omega_i$. This implies that $\tilde{u}(\Omega_i) = u_i(\Omega_i) = u(\Omega_i)$, since each simply connected $u_i(\Omega_i)$ (recall that $\Omega_i$ is simply connected) is according to the Jordan curve-theorem completely determined by its boundary curve.

Moreover, since $u$ is injective, the $u(\Omega_i)$ are mutually disjoint and since $u$ satisfies Lusin's N-condition $\bigcup_{i \in \mathbb{N}} u(\Omega_i)$ has full measure in $u(\Omega)$. Now, as $u$ fulfills \eqref{Ciarlet-Necas}, the claim follows.
\end{proof}

Finally let us mention that homogeneous gradient Young measures with support in quasiconformal matrices can be generated by quasiconformal maps:

\begin{theorem}[adapted from Theorem 1.5 in \cite{AstalaFaraco}]
\label{homYM}
Let $\nu$ be a homogeneous $W^{1,2}$-gradient Young measure with support contained in $\RK$. Then $\nu$ can be generated by a sequence of gradients of (uniformly) $K$-quasiconformal homeomorphisms $\{y_k\}_{k \in \mathbb{N}} \subset \mathcal{QC}(B_R(x_0); \mathbb{R}^{2 \times 2})$ for any $x_0 \in \R^2$ and $R > 0$.
\end{theorem}

Let us remark that this theorem is formulated in \cite{AstalaFaraco} in the following way: There exists a sequence of quasiconformal mappings $F_k : \mathbb{R}^2 \to \mathbb{R}^2$ such that the restriction of their gradients to the unit ball generates $\nu$. By a linear transformation of variables, we see that the gradients can be restricted to a ball of any radius and by translation the midpoint of the ball is arbitrary as well.

\section{Proofs of the main theorems}\label{Proofs}

\mbox{}

\begin{proof}(of Theorem \ref{THM1} -- characterization of quasiconformal gradient Young measures)

For the necessity, take a sequence $\{y_k\}_{k \in \mathbb{N}}$ of (uniformly) $K$-quasiconformal mappings converging weakly to $y(x)$ in $\mathcal{QC}(\Omega; \R^2)$. Clearly, $\{y_k\}_{k \in \mathbb{N}}$ generates a family of gradient Young measures $\nu_x \in \mathcal{GY}^2(\Omega; \mathbb{R}^{2 \times 2}$). Moreover,  $\nu_x$ is supported on the set $\bigcap_{l=1}^\infty\overline{\{\nabla y_k(x);\ k\ge l\}}$  (cf.\@ \cite{balder,valadier}) for almost all $x\in\O$; i.e., $\nu_x$ is supported on $\R^{2 \times 2}_K$.  Finally, the equality \eqref{firstmoment0} follows from the fundamental theorem of Young measures (cf.\@ e.g.\@ \cite[Theorem 6.2]{pedregal}). 

As for the sufficiency, we rely on a technique of partitioning the domain $\Omega$, that is routinely used in the analysis of gradient Young measures (cf.\@ \cite[Proof of Theorem~6.1]{k-p1}), on the result from \cite{AstalaFaraco} formulated in Theorem \ref{homYM} and importantly, on our novel cut-off technique that is presented in Section \ref{sect-cutOff}. 

Take $\nu \in \mathcal{GY}^2(\Omega; \R^{2 \times 2})$ and $y \in \mathcal{QC}(\Omega; \R^2)$ according to \eqref{supp} and \eqref{firstmoment0}. We aim to construct a sequence $\{y_k\}_{k\in\N}\subset \mathcal{QC}(\Omega; \R^2)$ converging weakly in $\mathcal{QC}(\Omega; \R^2)$ to $y(x)$, satisfying
\begin{equation}
\lim_{k\to\infty} \int_\O v(\nabla y_k(x))g(x)\,\md
x=\int_{\O} \int_{\R^{2\times
2}}v(s)\nu_x(\md s)g(x)\,\md x\  
\label{seek}
\end{equation}
for
all $g\in \Gamma$ and any $v\in S$, where
$\Gamma$ and $S$ are countable dense subsets of $C(\overline{\O})$ and
$C(\mathbb{R}^{2 \times 2}_+)$, respectively. In fact, we may fix $g$ and $v$ for the moment and once the generating sequence is found, rely on a diagonalization argument. 

We shall proceed, roughly, as follows: We cover $\Omega$ by small balls $B_{\varepsilon_{ik}}(a_{ik})$, the exact type of covering is given by an approximation of the integral on the right hand side of \eqref{seek} by suitable ``Riemann-sums'' in \eqref{79}. On each of these small balls the Young measure is roughly homogeneous, i.e. $\nu= \nu_{a_{ik}}$, with $\nabla y(a_{ik})$ being its first moment. For such a measure we may find a quasiconformal generating sequence due to Lemma \ref{homYM}. The idea is now to patch all these generating sequences defined on the small balls to obtain the final generating sequence. However, in order for the patched function to be really quasiconformal, we need to assure that all the generating sequences have the same boundary data; in fact, we will set them to be $y(x)$ on the boundary. To achieve this, we rely on Proposition \ref{prop-cutOff} but as a prerequisite we need the generating sequences (together with their inverses) to be locally uniformly close to $y(x)$. To assure this we rely on two ingredients: First, we know that the generating sequences converge weakly in $\mathcal{QC}(\Omega; \R^2)$, and thus are locally uniformly close, to $\nabla y(a_{ik})x$ (the same is true for the inverse). Second, since $y$ is differentiable, $y(a_{ik})+ \nabla y(a_{ik})x$ is uniformly close to $y(x)$ on the $B_{\varepsilon_{ik}}(a_{ik})$ (and similarly also for the inverse). 

Let us give the details of the proof: Since $y \in \mathcal{QC}(\Omega; \R^2)$, we know from the Gehring-Lehto theorem (cf.\@ \cite[Section 3.3]{AstalaIwaniec}) that it is differentiable in $\O$ outside a set of measure zero called $N$; in addition, we may assume that $\nabla y$ is finite outside of $N$. Also, $y^{-1}: y(\Omega) \to \Omega$ is differentiable almost everywhere and we may without loss of generality assume that the images of all points where the inverse of $y$ is not differentiable lie in $N$ because $y^{-1}$ maps null sets to null sets.  
Therefore, we find for every $a\in\O\setminus N$ and every $k > 0$ numbers $r_k(a)>0$ such that for any $0<\varepsilon < r_k(a)$ we have
\be\label{derivative1}
\left| \frac{y(x)-y(a)}{\varepsilon}-\nabla y(a)\left(\frac{x-a}{\varepsilon}\right)\right| \le \frac{1}{k}\  \qquad \forall x \in B_\varepsilon(a)
\ee
and also
\be\label{derivative2}
\left|\frac{y^{-1}(z)-y^{-1}(y(a))}{\varepsilon}-(\nabla y(a))^{-1}\left(\frac{z-y(a)}{\varepsilon}\right)\right| \le \frac{1}{k}\  \qquad \forall z \in \nabla y(a) B_{\eps}(0)+y(a)
\ee
Notice that in the second inequality we used that $\nabla y^{-1}(y(a)) = (\nabla y(a))^{-1}$.

Now, we perform the above announced ``suitable partitioning'' of the domain $\Omega$ by relying on \cite[Lemma~7.9]{pedregal}. Following this lemma, we can find $a_{ik}\in\O\setminus N$, $\varepsilon_{ik}\le  r_k(a_{ik})$ such that for all $v\in S$ and all $g\in \Gamma$
\be \label{79} \lim_{k\to\infty}\sum_i { \int_{B_{\varepsilon_{ik}}(a_{ik})}\overline{
V}(a_{ik})g(a_{ik}) \d x}= \int_\O \overline{
V}(x)g(x)\,\md x\ ,\ee
where
$$\overline{
V}(x):=\int_{\R^{2 \times 2}} v(s)\nu_x(\md s)\ .$$

It is well known (see e.g.\@ \cite[Proposition~8.18]{pedregal}), that $\nu_{a_{ik}}$ is a homogeneous $W^{1,2}$-gradient Young measure with $\nabla y(a_{ik})$ being its first moment; due to \eqref{supp}, we may assume that $\nu_{a_{ik}}$ is supported on $\RK$ and because the 
Jacobian of a quasiconformal mapping is strictly positive a.e.\@ (cf.\@ \cite[Section 3.7]{AstalaIwaniec}), we may also assume that $\det \nabla y(a_{ik}) > 0 $. Thus, in view of Theorem \ref{homYM}, this measure can be generated by gradients of a sequence of $K$-quasiconformal maps denoted $\{y^{ik}_j\}_{j\in\N}$. In other words we have that

\be\label{imp14} 
\lim_{j\to\infty} \int_{B_1(0)} v(\nabla y_j^{ik}(x))g(x)\,\md x=\overline{V}(a_{ik})\int_{B_1(0)} g(x)\,\md x\ 
\ee
and, in addition, $\{y^{ik}_j\}_{j \in \mathbb{N}}$ converges weakly in $\mathcal{QC}(B_1(0); \R^2)$ to the map $x \mapsto \nabla y(a_{ik})x$ for $j\to\infty$. In view of Lemma \ref{lem:HighInt} we know that $\{y^{ik}_j\}_{j \in \mathbb{N}}$ converges also locally uniformly to $x \mapsto \nabla y(a_{ik})x$. Moreover, $\{[y^{ik}_j]^{-1}\}_{j \in \mathbb{N}}$ is also a sequence of $K$-quasiconformal maps { and thus, due to Lemma \ref{lem:HighInt},} converges locally uniformly to some map $w(z)$. It is easy to identify $w(z) = (\nabla y(a_{ik}))^{-1}z$.  Indeed, take some arbitrary $x \in B_1(0)$ and $j$ large enough so that for some $\delta > 0$ we have $y^{ik}_j(x) \in B_\delta (\nabla y(a_{ik}) x)$ and  $B_{2\delta} (\nabla y(a_{ik}) x)\subset y^{ik}_j(B_1(0))$ for all such $j$. Then, $[y^{ik}_j]^{-1}(z)$ converges uniformly to $w(z)$ on $B_\delta (\nabla y(a_{ik}) x)$ and so { $z= y^{ik}_j ([y^{ik}_j]^{-1}(z)) \to \nabla y(a_{ik})(w(z))$}, in other words $w(z) =  (\nabla y(a_{ik}))^{-1}(z)$.  Thus, we may, owing to Proposition \ref{prop-cutOff}, assume that $y^{ik}_j(x) = \nabla y(a_{ik})(x)$ on $\partial B_1(0)$.

Further, consider for $k\in\mathbb{N}$, the rescaled functions $y_k$ defined through 
$$
y_k(x): = 
 y(a_{ik})+\varepsilon_{ik}y^{ik}_j\left(\frac{x-a_{ik}}{\varepsilon_{ik}}\right)  \qquad \forall x \in B_{\varepsilon_{ik}}(a_{ik})
$$
where $j=j(i,k)$ will be chosen later. Note that the above formula defines  $y_k$  almost everywhere in $\O$. Note also that the sequence $\{y_k\}_{k \in \mathbb{N}}$ is {$\kappa(K)$-quasiconformal (the quasiconformality constant might have changed due to Proposition \ref{prop-cutOff})} on each ball $B_{\varepsilon_{ik}}(a_{ik})$ and each function maps this ball to $\nabla y(a_{ik}) B_{\eps_{ik}}(0)+y(a_{ik})$ (due to the fixed boundary data).

We now show that on any compact subset of $B_{\varepsilon_{ik}}(a_{ik})$ the function $y_k$ is uniformly close to $y$ and the same holds for the inverses for $k$ large enough. Take some $x_0 \in B_1(0)$ and a radius $R$, such that $B_{2R}(x_0) \subset B_1(0)$; then we have that
\begin{align*}
\|y(x)&-y_k(x)\|_{L^\infty(B_{\varepsilon_{ik} R} (a_{ik}+\varepsilon_{ik}x_0); \R^2)} = \left\|y(x)-y(a_{ik})-\varepsilon_{ik}y^{ik}_j\left(\frac{x-a_{ik}}{\varepsilon_{ik}}\right) \right \|_{L^\infty(B_{\varepsilon_{ik} R} (a_{ik}+\varepsilon_{ik}x_0); \R^2)}  \\ &\leq \left\|y(x)-y(a_{ik})-\varepsilon_{ik}\nabla y(a_{ik})\left(\frac{x-a_{ik}}{\varepsilon_{ik}}\right)\right\|_{L^\infty(B_{\varepsilon_{ik} R} (a_{ik}+\varepsilon_{ik}x_0); \R^2)} \nonumber\\
&+\varepsilon_{ik}\left\|\nabla y(a_{ik})\left(\frac{x-a_{ik}}{\varepsilon_{ik}}\right)-y^{ik}_j\left(\frac{x-a_{ik}}{\varepsilon_{ik}}\right)\right\|_{L^\infty(B_{\varepsilon_{ik} R} (a_{ik}+\varepsilon_{ik}x_0); \R^2)} \leq \frac{2 \varepsilon_{ik}}{k} 
\end{align*}
if $j$ is large enough compared to $k$ and $i$ (at this point we choose $j$\footnote{{ Notice that $j$ may also depend on the chosen radius $R$, since we have only a locally uniform convergence of  $\{y^{ik}_j\}_{j \in \mathbb{N}}$ to $\nabla y(a_{ik})x$, but this is all what is needed for Proposition \ref{prop-cutOff}}}.), so that we can rely on the locally uniform convergence of  $\{y^{ik}_j\}_{j \in \mathbb{N}}$ to $\nabla y(a_{ik})x$ for $j\to\infty$ due to Lemma \ref{lem:HighInt}. 
Notice that by precomposing with a similarity mapping (which does not change the $K$-quasiconformality), this means that
$$
\|y(a_{ik}+\varepsilon_{ik}x)-y_k(a_{ik}+\varepsilon_{ik}x)\|_{L^\infty(B_R(x_0); \R^2)} \leq \frac{2}{k};
$$
i.e. the two maps are locally uniformly close to each other on the unit ball. For convenience, let us denote $\tilde{y}(x) = y(a_{ik}+\varepsilon_{ik}x)$ and $\tilde{y}_k(x) = y_k(a_{ik}+\varepsilon_{ik}x)$; computing the inverse maps gives
\begin{align*}
\tilde{y}^{-1}(z) = \frac{y^{-1}(z)-a_{ik}}{\varepsilon_{ik}}, \qquad \quad
\tilde{y}_k^{-1}(z)  = [y^{ik}_j]^{-1}\left(\frac{z-y(a_{ik})}{\varepsilon_{ik}}\right).
\end{align*}
Then for any point $z_0$ and any $\tilde{R} > 0$ such that $B_{2\tilde{R}}(z_0) \subset (\tilde{y}(B_1(0)) \cap  \tilde{y}_k(B_1(0))) \subset \nabla y(a_{ik}) B_{\eps_{ik}}(0)+y(a_{ik})$, we have that 
\begin{align*}
\|\tilde{y}^{-1}(z)&-\tilde{y}_k^{-1}(z)\|_{L^\infty(B_{\tilde{R}}(z_0); \R^2)} = \left\| \frac{y^{-1}(z)-a_{ik}}{\varepsilon_{ik}} - [y^{ik}_j]^{-1}\left(\frac{z-y(a_{ik})}{\varepsilon_{ik}}\right)\right\|_{L^\infty(B_{\tilde{R}}(z_0); \R^2)} \\& \leq \left\| [y^{ik}_j]^{-1}\left(\frac{z-y(a_{ik})}{\varepsilon_{ik}}\right) -   (\nabla y(a_{ik}))^{-1}\left(\frac{z-y(a_{ik})}{\varepsilon_{ik}}\right) \right\|_{L^\infty(B_{\tilde{R}}(z_0); \R^2)} \\&+ \left\| \frac{y^{-1}(z)-a_{ik}}{ \varepsilon_{ik}} - (\nabla y(a_{ik}))^{-1}\left(\frac{z-y(a_{ik})}{\varepsilon_{ik}}\right)\right\|_{L^\infty(B_{\tilde{R}}(z_0); \R^2)} \leq \frac{2}{k}
\end{align*}
due to \eqref{derivative2} and by enlarging $j$ if necessary. 

This puts us again into the situation of Proposition \ref{prop-cutOff}, so that we can modify $\tilde{y}_k$ to have the same trace as  $\tilde{y}$ on the boundary of  $B_1(0)$. By pre-composing again with the similarity $\frac{x-a_{ik}}{\varepsilon_{ik}}$ we thus obtain a modification of $y_k$ that has the same boundary values as $y(x)$ on $B_{\varepsilon_{ik}}(a_{ik})$. Let us call this modification $\bar{y}_k$. Finally let us set
$$
u_k(x) = \begin{cases} \bar{y}_k(x) &\text{if $x \in B_{\varepsilon_{ik}}(a_{ik})$}, \\
y(x) &\text{else}
\end{cases}
$$
and note that by the gluing Lemma \ref{lem:Gluing} $u_k$ is a sequence of {$\bar{K}$-\emph{quasiconformal} maps (again the quasiconformality constant might have changed due to Proposition \ref{prop-cutOff})} (i.e.\@ in particular \emph{homeomorphisms}). 

To see that $\{u_k\}$ generates $\nu_x$, we proceed in the same way as in \cite[Proof of Th.~6.1]{k-p1}. Indeed, 
by a diagonalization argument (relying on the fact that $\Gamma$ and $S$ are countable), we enlarge $j=j(i,k)$ if necessary so that \be\label{line} \left|\varepsilon_{ik}^n\int_{B_1(0)}
g(a_{ik}+\varepsilon_{ik}y)v(\nabla y_j^{ik}(y))\,\md y -\bar
V(a_{ik})\int_{B_{\varepsilon_{ik}}(a_{ik})}g(x)\,\md
x\right|\le\frac{1}{2^ik}\ .\ee
for all $(g,v) \in \Gamma \times S$. By summing and in view of \eqref{79} and \eqref{line}  we get that
$$
\lim_{k\to \infty} \int_\O g(x)v(\nabla y_k(x))\,\md x=\int_\O\int_{\R^{2\times 2}} v(s)\nu_x(\md s)g(x)\, \md x\ .$$
Hence, we can pick a sub-sequence of $\{\nabla y_k\}_k$ generating $\nu$. The measure $\nu$ is also generated by $\{\nabla u_k\}_k$ because the difference of both sequences vanishes in measure.
\end{proof}

In order to show Theorems \ref{THM2} and \ref{THM3}, we first need to establish that sequences that converge weakly in $\mathcal{QC}(\Omega; \R^2)$ to affine maps and have affine boundary data are actually non-concentrating. This is content of the following Lemma:

\begin{lemma}[Non-concentration of a sequence of quasiconformal maps]
\label{lem:non-conc}
Let $A \in \R^{2 \times 2}_+$, $b \in \R^2$ and $\{y_k\}_{k \in \mathbb{N}} \subset \mathcal{QC}(\Omega; \R^2)$ be a sequence { $\kappa(K)$-quasiconformal maps} weakly converging in $\mathcal{QC}(\Omega; \mathbb{R}^2)$ to the quasiconformal map $x \mapsto A x+b$ with $y_k(x) = Ax+b$ on $\partial \Omega$ for all $k\in\N$. Then, at least for a non-relabelled subsequence, $\{|\nabla y_k|^{ p}\}_{k \in \mathbb{N}}$ as well as $\{(\mathrm{det} \nabla y_k)^{- q}\}_{k \in \mathbb{N}}$ { for any $p < \frac{2\kappa(K)}{\kappa(K)-1}$ and $q< \frac{1}{\kappa(K)-1}$} are weakly converging in $L^1(\Omega)$.
\end{lemma}

\begin{proof}

By passing, if necessary, to the maps $\tilde{y}_k(x) = A^{-1}(y_k(x)-b)$, there is no loss in generality by assuming that $y_k: \Omega \to \Omega$ and $y_k(x) = x$ on $\partial \Omega$. Clearly, such maps can be extended by the identity to $\mathcal{QC}_\mathrm{loc}(\R^2; \R^2)$, so for simplicity we shall denote these extensions by $\{y_k\}_{k \in \mathbb{N}}$ as well. 

To see that $\{|\nabla y_k|^{ p}\}_{k \in \mathbb{N}}$ converges weakly in $L^1(\Omega)$, we notice that the higher integrability obtained in Lemma \ref{lem:HighInt} holds locally in $\R^2$ and so in particular $\{|\nabla y_k|^{ p}\}_{k \in \mathbb{N}}$ is bounded in $L^{1+{\varepsilon}}(\Omega)$ { for all $0<\varepsilon<\frac{2\kappa(K)}{(\kappa(K)-1)p}-1$} and hence has a subsequence weakly converging in $L^1(\Omega)$.

{ Note that, since the target domain is fixed and $\Omega$ as well, the same holds for the inverses; i.e. $\{|\nabla y_k^{-1}|^{ p}\}_{k \in \mathbb{N}}$ has a weakly weakly converging subsequence in $L^1(\Omega)$ as long as $p < \frac{2\kappa(K)}{\kappa(K)-1}$.}

{
To show the same for $\{(\mathrm{det} \nabla y_k)^{-{ q}}\}_{k \in \mathbb{N}}$ observe that, by the change of variables formula,
\begin{align*}
\int_\Omega \frac{1}{\det(\nabla y_k(x))^q} \d x &= \int_\Omega \frac{1}{(\det (\nabla y_k(x))^{q+1}} \det \nabla y_k(x) \d x 
= \int_\Omega (\det \nabla y_k^{-1}(y_k(x)))^{q+1} \det \nabla y_k(x) \d x   
\\&= \int_{y_k(\Omega)} (\det \nabla y_k^{-1}(z))^{q+1} \d z \leq c
\int_{\Omega} |\nabla y_k^{-1}(z)|^{2q+2} \d z.
\end{align*}
But $2q+2  < \frac{2\kappa(K)}{\kappa(K)-1}$ if and only if $q< \frac{1}{\kappa(K)-1}$ and so, since we found an equi-integrable majorant,  $\{(\mathrm{det} \nabla y_k)^{- q}\}$ is weakly converging in $L^1(\Omega)$ (at least for a sub-sequence). 
}
 \end{proof}
{
\begin{remark}
Notice that the above proof can be easily modified in order to obtain that the sequence $\{\phi(\mathrm{det} \nabla y_k)\}_{k \in \mathbb{N}}$ is equi-integrable for some function $\phi: \R_+ \to \R_+$ as long as it satisfies $\phi(t) \leq ct^{-q}$ with $q< \frac{1}{\kappa(K)-1}$. 

Moreover, in elasticity more general $\phi$, that may blow up at infinity, are used in growth conditions as in $\eqref{Orho}$. Also such a function can be incorporated in our proof, as long as the blow up at infinity is $s$-growth with $s < \frac{\kappa(K)}{\kappa(K)-1}$; i.e. $\phi(t) \leq C(1+t^s)$ for $t\geq r$ with some $r$ fixed, and for $t< r$ we still have that $\phi(t) \leq ct^{-q}$. Indeed, then we may separate
\begin{align*}
\sup_k \int_{B(x_0,\delta)} \phi(\det(\nabla y_k(x))) \d x &\leq \sup_k \int_{B(x_0,\delta)\cap (\{x\mid\det(\nabla y_k(x)) < r\} } \!\!\!\!\!\!\!\!\!\!\!\!\!\!\!\!\!\!\!\!\!\!\!\!\!\!\!\!\!\!\!\!\!\!\!\!\!\!\!\!\!\!\!\!\!\!\!\! \phi(\det(\nabla y_k(x))) \d x +  \sup_k \int_{B(x_0,\delta)\cap (\{x\mid\det(\nabla y_k(x)) \geq r\} } \!\!\!\!\!\!\!\!\!\!\!\!\!\!\!\!\!\!\!\!\!\!\!\!\!\!\!\!\!\!\!\!\!\!\!\!\!\!\!\!\!\!\!\!\!\!\!\! \phi(\det(\nabla y_k(x))) \d x  \\ &\leq \sup_k \int_{B(x_0,\delta)\cap (\{x\mid\det(\nabla y_k(x)) < r\} } \!\!\!\!\!\!\!\!\!\!\!\!\!\!\!\!\!\!\!\!\!\!\!\!\!\!\!\!\!\!\!\!\!\!\!\!\!\!\!\!\!\!\!\!\!\!\!\! \phi(\det(\nabla y_k(x))) \d x +  \sup_k \int_{B(x_0,\delta)\cap (\{x\mid\det(\nabla y_k(x)) \geq r\} } \!\!\!\!\!\!\!\!\!\!\!\!\!\!\!\!\!\!\!\!\!\!\!\!\!\!\!\!\!\!\!\!\!\!\!\!\!\!\!\!\!\!\!\!\!\!\!\! c(1+ |\nabla y_k(x)|^{2s})\d x,
\end{align*}
and proceed in both terms as in the the above proof.
\end{remark}
}

We are now ready to prove that measures in $\mathcal{GY}^{\mathcal{QC}}(\Omega; \R^{2 \times 2})$ satisfy the stricter version of the Jensen inequality given in Theorem \ref{THM2}; Theorem \ref{THM3} will follow from this.

\begin{proof}(of Theorem \ref{THM2})

First, we realize that for any  $\nu\in \mathcal{GY}^\mathcal{QC}(\Omega; \mathbb{R}^{2 \times 2})$, the homogeneous measure $\mu:=\{\nu_a\}_{x\in\O}$ is in $\mathcal{GY}^\mathcal{QC}(\Omega; \mathbb{R}^{2 \times 2})$, too, for a.e. $a\in \O$. To see this, we follow  \cite[Theorem ~7.2]{pedregal}: Indeed, if gradients of the sequence $\{y_k\}\subset \mathcal{QC}(\Omega; \mathbb{R}^{2 \times 2})$ generate $\nu$ then, for almost all $a\in\O$, we  construct a localized sequence $\{j y_k(a+x/j)\}_{j,k\in\N}$ (note that this is just a pre/post-composition with a similarity so it cannot affect $K$-quasiconformality) whose gradients generate $\mu$ as $j,k\to\infty$. { Moreover, we see that if $\nu$ could be generated by a sequence of $K$-quasiconformal maps than the same is true for $\mu$ (because the localization does not affect the distortion).}

Fix some $a \in \Omega$ { and find a generating sequence of $K$-quasiconformal maps $\{\nabla y_k\}_{k\in\N} \subset \mathcal{QC}(\Omega; \mathbb{R}^2)$ for} $\mu=\{\nu_a\}_{x\in\O}\in\mathcal{GY}^{\mathcal{QC}}(\Omega; \mathbb{R}^{2 \times 2})$. Then, this sequence converges weakly in $\mathcal{QC}(\Omega; \mathbb{R}^2)$ to the affine map $x \mapsto {y(a) +} (\nabla y(a))x$; notice that, since $y$ is quasiconformal, we may assume that $\det(\nabla y(a)) > 0$.   

Using Proposition~\ref{prop-cutOff}, we can without loss of generality suppose that $y_k(x)={y(a) +}\nabla y(a)x$ if $x\in\partial\O$ { by which we obtain a $\kappa(K)$-quasiconformal generating sequence}. Moreover, due to Lemma \ref{lem:non-conc} $\{|\nabla y_k|^{p}\}$ and $\{(\det(\nabla y_k))^{-q}\}$ { (with $p < \frac{2\kappa(K)}{\kappa(K)-1}$ and $q< \frac{1}{\kappa(K)-1}$)} are weakly convergent in $L^1(\Omega)$. Therefore, we have
$$
|\O|\int_{\R^{2\times 2}} v(s)\nu_a(\md s) = \lim_{k\to\infty} \int_\O v(\nabla y_k(x))\,\md x \ge |\O| v(\nabla y(a)) \ .$$
for any $v \in \mathcal{E}{( K)}$ that is {$\kappa(K)$}-quasiconformally quasiconvex.
\end{proof}

\begin{proof}(of Theorem \ref{THM3})
For showing the weak lower semicontinuity, take a sequence of { $K$-quasiconformal} maps $\{y_k\}_{k \in \N}$ that converges weakly in $\mathcal{QC}(\Omega; \mathbb{R}^2)$ to $y$. We know that this sequence generates a measure $\nu \in \mathcal{GY}^\mathcal{QC}(\Omega; \mathbb{R}^{2 \times 2})$  and so we have from Theorem \ref{THM2}
$$
 \int_\Omega v(\nabla y(x)) \md x  \le\int_\Omega \int_{\R^{2 \times 2}} v(s)\nu_x(\md s) \d x \leq \liminf_{k \to \infty} \int_\Omega v(\nabla y_k) \md x
$$
for any $v \in \mathcal{E}$ that is quasiconformally quasiconvex.

For the opposite, { let us first realize that if $I(y):=\int_\O v(\nabla y(x))\,\d x$ is weakly lower semicontinuous along $K$-quasiconformal sequences then so is $I_{\Omega'}(y):=\int_{\Omega'} v(\nabla y(x))\,\d x$ for any Lipschitz domain $\Omega' \subset \Omega $. Indeed, consider first the case when $\Omega' = a+\delta\Omega$, i.e. $\Omega'$ is a scaled copy of $\Omega$. Take any $K$-quasiconformal sequence $\{y_k\}_{k \in \mathbb{N}}$ converging weakly to $y$  in $\mathcal{QC}(\Omega'; \mathbb{R}^2)$. Then $\{\frac{1}{\delta}y_k(a+\delta x)\}_{k \in \mathbb{N}}$ converges weakly to $\frac{1}{\delta}y(a+\delta x)$  in $\mathcal{QC}(\Omega; \mathbb{R}^2)$ (note that also the quasiconformality constant remained unchanged) and so
$$
\int_{\Omega'} v(\nabla y(x'))\,\d x' = \delta^2 \int_\O  v(\nabla y(a+\delta x))\,\d x \leq \liminf_{k \to \infty} \delta^2\int_\O  v(\nabla y_k(a+\delta x))\,\d x = \liminf_{k \to \infty} \int_{\Omega'} v(\nabla y_(x'))\,\d x'
$$
Now if $\Omega'$ is a general Lipschitz subdomain of $\Omega$, we may find for every $\varepsilon > 0$ a \emph{finite} collection of disjunct sets $\{(a_i+\delta_i\Omega)\}_i$ such that $\bigcup_i (a_i+\delta_i\Omega) \subset \Omega'$ and $|\Omega'\setminus \bigcup_i (a_i+\delta_i\Omega)| \leq \varepsilon$. Since any sequence $\{y_k\}_{k \in \mathbb{N}}$ of $K$-quasiconformal maps converging weakly to $y$  in $\mathcal{QC}(\Omega'; \mathbb{R}^2)$ is also weakly convergent to the same limit when restricted to $a_i+\delta_i\Omega$, we obtain that
$$
\sum_i \int_{a_i+\delta_i\Omega} v(\nabla y(x))  \leq \sum_i \liminf_{k\to\infty} \int_{a_i+\delta_i\Omega} v(\nabla y_k(x)) \d x \leq  \liminf_{k\to\infty} \sum_i \int_{a_i+\delta_i\Omega} v(\nabla y_k(x)) \d x \leq \int_{\Omega'} v(\nabla y_k(x)),
$$
and taking $\varepsilon \to 0$ yields the claim.
}

Take any { $K$-quasiconformal} $y \in \mathcal{QC}(\Omega; \mathbb{R}^2)$ such that $y(x) = Ax$ in $\partial \Omega$. Then this $y$ defines a homogeneous Young measure $\nu\in \mathcal{GY}^\mathcal{QC}(\Omega; \mathbb{R}^{2 \times 2})$ with $A$ being its first moment via setting  
$$\int_{\R^{2\times 2}}f(s)\nu(\d s):=|\O|^{-1}\int_\O f(\nabla y(x))\,\d x$$ 
for every $f$ in $\mathcal{E}{(K)}$. Let us find a generating sequence for $\nu$ consisting of gradients of { $K$-}quasiconformal maps $\{y_k\}_{k \in \mathbb{N}}$. { Notice that we can adjust the sequence to satisfy $y_k(x) = Ax$ on $\partial \Omega$ which may change the distortion to $\kappa(K)$; however, it also follows from Proposition \ref{prop-cutOff} that for any $\varepsilon > 0$ the functions $y_k$ are not modified on a suitable Lipschitz domain $\Omega_\varepsilon \subset \Omega$ with $|\Omega \setminus \Omega_\varepsilon| \leq \varepsilon$, for all $k$ large enough. Therefore the adjusted $y_k$ still are  $K$-quasiconformal on $\Omega_\varepsilon$.}  Recall that for such a sequence $\{|y_k|^p\}_{k \in \mathbb{N}}$ as well as $\{(\mathrm{det} \nabla y_k)^{-q}\}_{k \in \mathbb{N}}$ are weakly converging in $L^1(\Omega)$ { (with $p < \frac{2\kappa(K)}{\kappa(K)-1}$ and $q< \frac{1}{\kappa(K)-1}$)}. Also notice that $y_k \rightharpoonup Ax$ in $\mathcal{QC}(\Omega; \mathbb{R}^2)$ since $A$ is the first moment of $\nu$. 

{
Now, since  $I_\varepsilon(y):=\int_{\Omega_\varepsilon} v(\nabla y(x))\,\d x$ is weakly lower semicontinuous on  $\mathcal{QC}(\Omega_\varepsilon; \mathbb{R}^2)$ we get 
$$
(|\O|- |\Omega \setminus \Omega_\varepsilon|)v(A) = I_\varepsilon(Ax) \leq \liminf_{k\to\infty}I_\varepsilon(y_k) = \int_{\Omega_\varepsilon} \int_{\R^{2\times 2}}v(s)\nu(\d s)\d x = \int_{\Omega_\varepsilon} v(\nabla y(x))\,\d x \leq  \int_\O v(\nabla y(x))\,\d x \ ,$$
which, when passing with $\varepsilon \to 0$, shows that $v$ is $K$-quasiconformally quasiconvex.}
\end{proof}
 
{
\begin{proof}(of Theorem \ref{minimizers})
First, let us realize that there exist maps in $W^{1,2}(\Omega;\R^2)$, that satisfy the Ciarlet-Ne\v{c}as condition \eqref{Ciarlet-Necas} and coincide with $\eta$ on $\Gamma$, on which the functional $J$ is finite. This follows from the fact that $\eta$ is quasi-symmetric; hence, it can be extended to a quasiconformal mapping of $\Omega$ that coincides with $\eta$ on $\partial \Omega$ and so, in particular, also in $\Gamma$ (see Lemma \ref{lemma-extension} and the remarks below its proof). From Lemma \ref{lemma:C-N} we know that this map satisfies \eqref{Ciarlet-Necas} and because the function is quasiconformal its distortion $\frac{|\nabla y|^2}{\mathrm{det}(\nabla y)}$ is uniformly bounded.

Let us take a minimizing sequence of $J(y)$ denoted $\{y_k\}_{k \in \mathbb{N}} \subset W^{1,2}(\Omega;\R^2)$ that satisfies the Ciarlet-Ne\v{c}as condition  and coincides with $\eta$ on $\Gamma$.  Since $J$ has to be uniformly bounded along this sequence, i.e.
$
J(y_k) \leq C
$,
we have that $\det(\nabla y_k(x)) > 0$  a.e. on $\Omega$ for all $k \in \mathbb{N}$ and there exists a constant $K$ such that 
$$
\left\|\frac{|\nabla y|^2}{\mathrm{det}(\nabla y)} \right\|_{L^\infty(\Omega)} \leq K \qquad \forall k \in \mathbb{N};
$$
in other words the minimizing sequence is uniformly $K$-quasiregular and non-constant. Moreover, since each individual member of the sequence satisfies \eqref{Ciarlet-Necas} it is quasiconformal (see Lemma \ref{lemma:C-N}). 

Thus, we may select a weakly convergent subsequence of  $\{y_k\}_{k \in \mathbb{N}}$ (not relabeled) in $W^{1,2}(\Omega;\R^2)$ with the weak limit being $y \in W^{1,2}(\Omega;\R^2)$. Clearly, since $y$ coincides with $\eta$ on $\Gamma$, it is non-constant and hence the convergence is weak even in $\mathcal{QC}(\Omega; \R^{2 \times 2})$. Owing to Remark \ref{rem-convQC}, $y$ is $K$-quasiconformal whence $\det(\nabla y) > 0$ a.e. on $\Omega$. 

Due to the weak-lower semicontinuity theorem \ref{THM3}, we have that
$$
\int_\Omega v(\nabla y) \dd x \leq \liminf_{k \to \infty} \int_\Omega v(\nabla y_k) \dd x.
$$

It remains to show the distortion is weakly lower semi-continuous, i.e.
\begin{equation}
\label{dist-weak}
\left\|\frac{|\nabla y|^2}{\mathrm{det}(\nabla y)} \right\|_{L^\infty(\Omega)}\leq \liminf_{k \to \infty} \left\|\frac{|\nabla y_k|^2}{\mathrm{det}(\nabla y_k)} \right\|_{L^\infty(\Omega)},
\end{equation}
because then
$$
J(y) \leq \liminf_{k \to \infty} \int_\Omega v(\nabla y_k) \dd x + \varepsilon \liminf_{k \to \infty} \left\|\frac{|\nabla y_k|^2}{\mathrm{det}(\nabla y_k)} \right\|_{L^\infty(\Omega)} \leq \liminf_{k \to \infty} J(y_k),
$$
i.e. $y$ is the sought minimizer.

For showing \eqref{dist-weak}, we use Remark \ref{rem-convQC} that states that the weak limit of $Q$-quasiconformal mappings in $\mathcal{QC}(\Omega; \R^{2 \times 2})$ is also $Q$-quasiconformal. In fact, up to selecting another subsequence, we may assume that 
$$
\lim_{k \to \infty} \left\|\frac{|\nabla y_k|^2}{\mathrm{det}(\nabla y_k)} \right\|_{L^\infty(\Omega)} = Q;
$$
therefore, for any $\varepsilon > 0$, we know that for $k$ large enough $\left\|\frac{|\nabla y_k|^2}{\mathrm{det}(\nabla y_k)} \right\|_{L^\infty(\Omega)} \leq  Q + \varepsilon$ that is $\{y_k\}$ are $(Q+\varepsilon)$-quasiconformal. But then so is the limit $y$ and thus
$$ 
\left\|\frac{|\nabla y|^2}{\mathrm{det}(\nabla y)} \right\|_{L^\infty(\Omega)} \leq Q+\varepsilon.
$$
Finally, since $\varepsilon > 0$ was arbitrary, \eqref{dist-weak} is obtained.
\end{proof}
}

\section{Cut-off technique preserving for quasiconformal maps}
\label{sect-cutOff}

Within this section, we present our cut-off technique that preserves quasiconformality, which is the crucial ingredient to the proofs of Theorems \ref{THM1}-\ref{THM3}. 

\begin{proposition}\label{prop-cutOff}
  Let $\mathrm{diam}(\Omega) >> \varepsilon > 0$. Further let $y_k, y \in \mathcal{QC}(\Omega; \mathbb{R}^2)$  be $K$-quasiconformal. Then there exists a $\delta << \varepsilon$ that depends only on $y, K$ and $\varepsilon$ such that if $y_k, y$ satisfy  
\begin{equation}
  \|y-y_k\|_{L^\infty(B_R(x_0); \mathbb{R}^2)} \leq \delta \qquad \text{and} \qquad  \|y^{-1}-y_k^{-1}\|_{L^\infty(B_R(z_0); \mathbb{R}^2)} \leq \delta
  \label{needs-CutOff}
\end{equation}
   for all $x_0$ and ${R\geq\varepsilon}$ such that $B_{2R}(x_0) \subset \Omega$ and all $z_0$ and ${R\geq\max_{\abs{x-x_0}=\varepsilon}\abs{y(x)-y(x_0)}}$ such that $B_{2R}(z_0) \subset y(\Omega) \cap y_k(\Omega)$,\footnote{Note that the conditions from \eqref{needs-CutOff} hold for $k$ large enough if the sequence $\{y_k\}$ along with its inverses converge locally uniformly to $y$ and its inverse, respectively.} a $\kappa(K)$-quasiconformal function $\omega \in \mathcal{QC}(\Omega;\mathbb{R}^2)$ with the following properties can be constructed:
  \begin{enumerate}
    \item $
  \|y-\omega\|_{L^\infty(\Omega; \mathbb{R}^2)} \leq C(\varepsilon) \qquad \text{and} \qquad \|y^{-1}-\omega^{-1}\|_{L^\infty(y(\Omega); \mathbb{R}^2)} \leq C(\varepsilon),
$
  \item $\kappa(K)$ depends only on $K$,
    \item $\omega\mid_{\partial \Omega} = y\mid_{\partial\Omega}$,
    \item $\abs{\left\{x\in \Omega: y_k(x) \neq \omega\right\}} < C(\varepsilon)$,
  \end{enumerate}  
  where $C(\varepsilon) \rightarrow 0$ for $\varepsilon \rightarrow 0$.
\end{proposition}

We prove Proposition \ref{prop-cutOff} in the remainder of this section by explicitly constructing the sought function $\omega$. To do so, we will divide the domain $\Omega$ into three parts: An outer shell $\Omega_{outer}$, which includes all points of $\Omega$ close to $\partial \Omega$, the set $\Omega_{inner}$, consisting of the bulk of $\Omega$, and a small strip between the two sets, denoted $\Omega_{mid}$; cf. Construction \ref{const-outer} for a formal definition and Figure \ref{fig-PartitionOmega} for a better overview. Let us also note that even though the names for the partitions of $\Omega$ are lent from the situation when $\Omega$ is simply connected, simple connectivity is not needed in our proof.

We will simply set $\omega = y$ on $\Omega_{outer}$ to obtain the right boundary condition (Proposition \ref{prop-cutOff}, item 3) and $\omega = y_k$ on $\Omega_{inner}$ in order to satisfy condition Proposition \ref{prop-cutOff}, item 4. Finally, on the strip $\Omega_{mid}$ we will join the two parts using the Beurling-Ahlfors extension so that the resulting function still ends up in $\mathcal{QC}(\Omega; \mathbb{R}^2)$ with a quasiconformality constant depending only on $K$. However, as explained in the proof of Lemma \ref{lemma-extension}, in order to apply the Beurling-Ahlfors extension on a given domain we need to be able to transform it conformally to the half-plane. This, in particular, is not possible for $\Omega_{mid}$, since it is in general not simply connected and so we will further partition $\Omega_{mid}$ into squares on each of which the extension property can be used. Yet, then we have to define $\omega$ on edges of the squares which lie strictly in $\Omega_{mid}$ (so-called ``bridges'', denoted $G$ in Construction \ref{const-outer} and Figure \ref{fig-PartitionOmega}) in a quasisymmetric way with the quasisymmetry modulus $\eta$ determined only by $K$. This will form the heart of our construction and, in fact, the major part of the proof.
 
Let us start by giving a detailed description of the partition of the domain:

\begin{construction}[Partition of $\Omega$] \label{const-outer}
  Fix $\mathrm{diam}(\Omega) >> \varepsilon > 0$ and consider the grid of points $\alpha \in \varepsilon \cdot \Z^2$. Using this grid, we tile $\Omega$ into 
  $$S_{\alpha} := \left\{x \in \R^2\mid \alpha_1 \leq x_1 \leq \alpha_1+\varepsilon \wedge \alpha_2 \leq x_2 \leq \alpha_2+\varepsilon \right\} \cap \Omega$$ 
  and set
  \begin{align*}
  \Omega_{outer} &:= \bigcup \left\{ S_\alpha: \alpha\in \varepsilon\Z^2, \mathrm{dist}(S_\alpha, \partial \Omega) < 2\gamma \varepsilon \right\}, \\ 
  \Omega_{inner} &:= \bigcup \left\{ S_\alpha: \alpha\in \varepsilon\Z^2, S_\alpha \cap \Omega_{outer} = \emptyset \right\}, \\
  \Omega_{mid} &:= \Omega \setminus (\Omega_{inner} \cup \Omega_{outer}),
 \end{align*}
  where $\gamma$ is the smallest integer satisfying $\gamma \geq 1$ and $\eta(1/\gamma) \leq 1/4$ with $\eta$ being the local quasisymmetry modulus of $y$ and $y_k$ (notice that this function depends only on $K$ due to Lemma \ref{lemma-locQuasi}).  
  Furthermore, we denote by $G$ all grid-lines $(\alpha, \alpha + \unit_i \varepsilon) \subset \Omega_{mid}$ for $i \in \{1,2\}$. 
  
  Finally, 
   we set
\begin{equation}
\omega(x) := \begin{cases}y(x) &\mbox{ for } x \in \Omega_{outer} \\ y_k(x) &\mbox{ for } x\in \Omega_{inner} \end{cases},
\label{Omega-def1}
\end{equation}
\end{construction}

So, $\Omega_{outer}$ consists of all those squares that are close to $\partial \Omega$ ($\varepsilon$ is presumed to be small) and $\Omega_{mid}$ is essentially a one square deep row separating $\Omega_{inner}$ and $\Omega_{outer}$; we refer to Figure \ref{fig-PartitionOmega} for an illustration of the situation. 
\begin{figure}
\center{\includegraphics{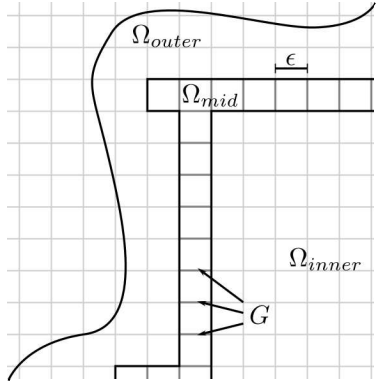}}
\caption{Partition of $\Omega$}
\label{fig-PartitionOmega}
\end{figure}

\begin{remark}
\label{remark-quasisym}
Note that $\gamma$ determining the distance of $\Omega_{mid}$ to the boundary is chosen in such a way that \eqref{needs-CutOff} is satisfied and $y$ as well as $y_k$ are $\eta$-quasisymmetric on each of the squares in $\Omega_{mid}$. For this, we need to verify that every such square lies in a ball $B_R(x_0)$ such that $B_{2R}(x_0) \subset \Omega$ and the image lies in another ball $B_R(z_0)$ such that $B_{2R}(z_0) \subset y(\Omega) \cap y_k(\Omega)$. 

Now the first part is easy to see. As for the image of the square, we consider first its image under $y$. Let $x_0$ be the midpoint of the square, then we realize that the image of the square under $y$ has to lie in $y(B_\varepsilon(x_0))$ which itself lies in a ball $B_{\max_{\{|x-x_0|=\varepsilon\}} |y(x_0) - y(x)|}(y(x_0))$. On the other hand, we know that $B_{\min_{\{|x-x_0|=\gamma \varepsilon\}} |y(x_0) - y(x)|}(y(x_0))$ is contained in $y(\Omega)$. Due to quasisymmetry, 
$$
\frac{\max_{\{|x-x_0|=\varepsilon\}} |y(x_0) - y(x)|}{\min_{\{|x-x_0|=\gamma \varepsilon\}} |y(x_0) - y(x)|} \leq \eta(1/\gamma).
$$
This shows that 
$$B_{2\max_{\{|x-x_0|=\varepsilon\}} |y(x_0) - y(x)|}(y(x_0)) \subset B_{4\max_{\{|x-x_0|=\varepsilon\}} |y(x_0) - y(x)|}(y(x_0)) \subset y(\Omega).$$
Analogously, we obtain that
$$B_{2\max_{\{|x-x_0|=\varepsilon\}} |y_k(x_0) - y_k(x)|}(y_k(x_0)) \subset B_{4\max_{\{|x-x_0|=\varepsilon\}} |y_k(x_0) - y_k(x)|}(y_k(x_0)) \subset y_k(\Omega).$$

Finally, we need to verify that the ball of radius $B_{2\max_{\{|x-x_0|=\varepsilon\}} |y(x_0) - y(x)|}(y(x_0)) \subset y_k(\Omega)$ and vice versa. For this we choose 
\begin{equation}
\delta \leq \min_{x_0} \max_{\{|x-x_0|=\varepsilon\}} |y(x_0) - y(x)|;
\label{delta_choice}
\end{equation}
note that there is a finite number of $x_0$'s (depending on $\varepsilon$) so that the minimum can be found and is positive. But then, since
$$
|y(x_0) - y(x)|-2\delta\leq |y_k(x_0) - y_k(x)| \leq |y(x_0) - y(x)|+2\delta,
$$
we have that $B_{2\max_{\{|x-x_0|=\varepsilon\}} |y_k(x_0) - y_k(x)|}(y_k(x_0)) \subset B_{4\max_{\{|x-x_0|=\varepsilon\}} |y(x_0) - y(x)|}(y(x_0))$ as well as \\ $B_{2\max_{\{|x-x_0|=\varepsilon\}} |y(x_0) - y(x)|}(y(x_0)) \subset B_{4\max_{\{|x-x_0|=\varepsilon\}} |y_k(x_0) - y_k(x)|}(y_k(x_0))$ which shows the claim.
\end{remark}

Note that \ref{Omega-def1} defines $\omega$ everywhere except for $\Omega_{mid}$. It is trivial to see that $\omega$ fulfills conditions {3 and 4} in Proposition \ref{prop-cutOff} and that item 1 and 2 in this proposition hold so far as $\omega$ is defined. Furthermore, if $\delta < \varepsilon/5$ and \eqref{delta_choice} hold $\partial\Omega_{inner}$ and $\partial\Omega_{outer}\setminus\partial\Omega$ will not intersect (see Lemma \ref{lemma-constructionProperties} below), which makes $\omega$ injective so far as defined.

It remains to define $\omega$ on $\Omega_{mid}$, which is the non-trivial part of the construction, however. As outlined above, we first define $\omega$ on the grid segments of $G$. This can be done for each grid line in $G$ independently, and so, since all the cases are essentially equivalent, we may turn our attention to the  single line segment $(\alpha,\alpha+ \varepsilon \unit_1)$ with $\alpha \in \Omega_{outer}$ and $\alpha + \varepsilon \unit_1 \in\Omega_{inner}$. 

On this consider the following construction:

\begin{construction}[Building a bridge on $G$] \label{const-inner}
  Define
  \[r := \min\left\{ s > 0: y(\partial B_s(\alpha)) \cap y_k(\partial B_s(\alpha+ \varepsilon \unit_1)) \neq \emptyset \right\}\]
  and take some $z_0\in y(\partial B_r(\alpha)) \cap y_k(\partial B_r(\alpha+\varepsilon \unit_1))$. Then we define the affine functions
\begin{align*}
\phi_1(x) &= \frac{2}{\varepsilon}(y^{-1}(z_0)-\alpha)(x_1-\alpha_1) + \alpha + \unit_2 (x_2-\alpha_2), \\
  \phi_2(x ) &= \frac{2}{\varepsilon}(y_k^{-1}(z_0) - (\alpha + \varepsilon \unit_1) ) (\alpha_1 + \varepsilon - x_1) +\alpha + \varepsilon \unit_1+ \unit_2 (x_2-\alpha_2),
\end{align*}
that are constructed in such a way that 
\begin{align*}
\phi_1(\alpha+s \unit_2) &= \alpha+s \unit_2 \quad \text{ and } \quad \phi_2(\alpha+\varepsilon \unit_1+s \unit_2) = \alpha+\varepsilon \unit_1+s \unit_2  \qquad \forall s\in\R, \\
\phi_1\left(\alpha+\frac{\varepsilon}{2} \unit_1\right) &= y^{-1}(z_0) \quad \text{ and } \quad \phi_2\left(\alpha+\frac{\varepsilon}{2} \unit_1\right) = y_k^{-1}(z_0),
\end{align*}
see Figure \ref{fig-bridge} for an illustration of the situation.

We now define
  \[\omega_\alpha: (0,\varepsilon) \rightarrow \R^2, s \mapsto \begin{cases} y\circ \phi_1(\alpha + s \unit_1) & \mbox{ for } s < \varepsilon/2 \\ y_k\circ \phi_2(\alpha + s \unit_1) & \mbox{ for } s \geq \varepsilon/2 \end{cases}\]
\end{construction}

\begin{remark}
Note that the minimum of the set $\left\{ s > 0: y(\partial B_s(\alpha)) \cap y_k(\partial B_s(\alpha+ \varepsilon \unit_1)) \neq \emptyset \right\}$ needed in Construction \ref{const-inner}, can be found since the set {is closed and bounded from below}. In our arguments we will also sometimes make use of the equivalent characterizations
\begin{align*}
r &= \min\big\{ s > 0: \overline{y(B_s(\alpha))} \cap \overline{y_k(B_s(\alpha+ \varepsilon \unit_1))} \neq \emptyset \big\},\\
r &= \max\big\{ s > 0: y(B_s(\alpha)) \cap y_k(B_s(\alpha+\varepsilon \unit_1)) = \emptyset \big\}.
\end{align*}
\end{remark}

The function $\omega_\alpha$ defined in Construction \ref{const-inner} seems to be a promising candidate for the sought ``bridge function'', i.e.\@ a definition of $\omega$ on $G$. However, as defined, $\omega_\alpha$, will not necessarily be quasisymmetric which is essential for Lemma \ref{lemma-extension}. Nevertheless, it will turn out in Lemma \ref{lemma-quasicircle} below that, in fact, the image of $\omega_\alpha$ is at least a segment of a quasicircle and so the sought bridge on the segment  $[\alpha,\alpha + \varepsilon \unit_1] \subset \Omega_{{mid}}$ will be a reparametrization of $\omega_\alpha$.

  

\begin{figure}
\center{\includegraphics[width=0.5\paperwidth]{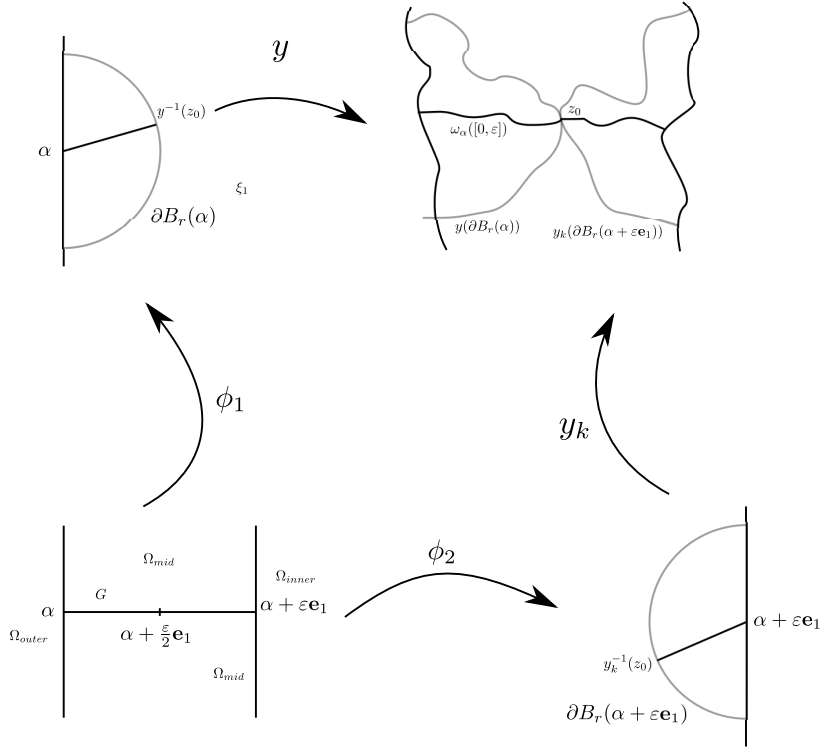}}
\caption{Connecting $\Omega_{outer}$ and $\Omega_{inner}$ on $G$ (deformations are exaggerated in scale)}
\label{fig-bridge}
\end{figure}

Before proceeding, we will show the well-definedness of Construction \ref{const-inner} as well as some bounds on the various quantities involved.
\begin{lemma} \label{lemma-constructionProperties}
 Let $\delta$ satisfy \eqref{delta_choice} as well as $\delta < \frac{\varepsilon}{5}$. Let $y, y_k$ be as in Proposition \ref{prop-cutOff} and fulfill \eqref{needs-CutOff}. Then the quantities found in Construction \ref{const-inner} satisfy the following:
  \begin{enumerate}
    \item $\frac{\varepsilon}{2}-\delta/2 < r < \frac{\varepsilon}{2}+ \delta/2$ 
    \item $\abs{y^{-1}(z_0)- (\alpha +\frac{\varepsilon}{2}\unit_1 )} < \sqrt{7\varepsilon \delta}$ and $\abs{y_k^{-1}(z_0)- (\alpha+\frac{\varepsilon}{2}\unit_1 )} < \sqrt{7\varepsilon \delta}$, 
    \item The functions $\phi_i$ are bi-Lipschitz with constants $L_{\phi_i}< 1/(1-\frac{\delta}{\varepsilon})$, $i=1,2$.
  \item The function $\omega_\alpha$ is well-defined, injective and $\omega_\alpha((0,\varepsilon))$ does neither intersect $y(\Omega_{outer})$, $y_k(\Omega_{inner})$ nor the images of similarly constructed $\omega_\beta$ on any of the other edges in $G$.
  \item The function $\omega_\alpha$ is $\tilde{\eta}$-quasisymmetric in $[0, \varepsilon/2]$ as well $[\varepsilon/2, \varepsilon]$ with $\tilde{\eta}$ only dependent on $\eta$.
  \end{enumerate}
\end{lemma}

\begin{proof} \mbox{}

  \begin{enumerate} \item Let us look at the bounds for $r$ first. We know that $\abs{y^{-1}(z_0)-y_k^{-1}(z_0)} < \delta$ but $y^{-1}(z_0)$ and $y_k^{-1}(z_0)$ lie on a circle of equal radius centered at $\alpha$ and $\alpha + \varepsilon \unit_1$, respectively. Therefore, it has to hold that $r > \varepsilon/2-\delta/2$. 
  
 For the upper bound, consider the point $\tilde{z}:=y\left(\alpha + \frac{\varepsilon+\delta}{2}\unit_1\right)$ in the image. We know that
  \[\abs{\left(\alpha + \frac{\varepsilon+\delta}{2}\unit_1\right) - y_k^{-1}(\tilde{z})} = \abs{y^{-1}(\tilde{z}) - y_k^{-1}(\tilde{z})}< \delta.\]
  But then $y_k^{-1}(\tilde{z}) \in B_\delta\left(\alpha + \frac{\varepsilon+\delta}{2}\unit_1\right) \subset \overline{B_{\varepsilon/2+\delta/2}(\alpha + \varepsilon \unit_1)}$, so $\tilde{z} \in \overline{y_k(B_{\varepsilon/2+\delta/2}(\alpha + \varepsilon \unit_1))} \cap \overline{y(B_{\varepsilon/2+\delta/2}(\alpha))}$. So this means that $\varepsilon/2+\delta/2$ is an upper bound for $r$.
  
  \item Using the estimates for $r$, we can now bound the distances $\abs{y^{-1}(z_0)- \left(\alpha +\frac{\varepsilon}{2}\unit_1 \right)}$ and $\abs{y_k^{-1}(z_0)- \left(\alpha+\frac{\varepsilon}{2}\unit_1 \right)}$; due to symmetry of the two we shall just show the latter. Again, we start with $\abs{y^{-1}(z_0)-y_k^{-1}(z_0)} < \delta$  and so, using the upper bound on $r$, we have on one hand
  \[\abs{\alpha-y_k^{-1}(z_0)} \leq \abs{\alpha - y^{-1}(z_0)} + \abs{y^{-1}(z_0)-y_k^{-1}(z_0)} < r+\delta < \frac{\varepsilon}{2}+\frac{3\delta}{2},\]
  and on the other hand 
  $$\abs{(\alpha+\varepsilon \unit_1)-y_k^{-1}({z_0})} = r < \frac{\varepsilon}{2}+\frac{\delta}{2} \leq \frac{\varepsilon}{2}+\frac{3\delta}{2}.$$
  So $y_k^{-1}({z_0})$ and $\alpha + \frac{\varepsilon}{2}\unit_1$ are both inside the ``lens-like'' intersection of the two circles $B_{\varepsilon/2+3/2\delta}(\alpha) \cap B_{\varepsilon/2+3/2\delta}(\alpha + \varepsilon \unit_1)$ (cf. Figure \ref{figure-lens}), whose diameter is bounded by $\sqrt{7 \varepsilon \delta}$. 
  
  \item It is easy to calculate the gradient of $\phi_1$:
  \[D\phi_1 = \left(\frac{2}{\varepsilon}y^{-1}(z_0)-\alpha \middle| \unit_2\right),\]
  which has eigenvalues $1$ and $\frac{2}{\varepsilon}(y^{-1}(z_0)-\alpha)_1$, so 
  $$L_{\phi_1} = \max\left\{\frac{2}{\varepsilon}(y^{-1}(z_0)-\alpha)_1, \frac{1}{\frac{2}{\varepsilon}(y^{-1}(z_0)-\alpha)_1}\right\} \leq \max\left\{1+\frac{\delta}{\varepsilon},\frac{1}{1-\frac{\delta}{\varepsilon}}\right\} = \frac{1}{1-\frac{\delta}{\varepsilon}}.$$ 
  The same works for $\phi_2$. Also note that $\delta < \frac{\varepsilon}{5}$ implies $L_{\phi_i} < \frac{5}{4}$.
  
  \item By the bounds from 1., we know that $\omega_\alpha$ is well-defined. Now injectivity of $\omega_\alpha$ restricted to each of the intervals $(0,\varepsilon/2)$ and $(\varepsilon/2,\varepsilon)$ results directly from the fact that all of the constituent functions $y,y_k,\phi_1$ and $\phi_2$ are injective. Furthermore we have that $\omega_\alpha((0,\varepsilon/2)) \subset y(B_r(\alpha))$ and $\omega_\alpha((\varepsilon/2,\varepsilon)) \subset y_k(B_r(\alpha+\varepsilon \unit_1))$. But since $y(B_r(\alpha)) \cap y_k(B_r(\alpha+\varepsilon \unit_1)) = \emptyset$ this implies injectivity of $\omega_\alpha$ on all of $(0,\varepsilon)$. 
  
  By our bounds on the positions of $y^{-1}(z_0)$ and $y_k^{-1}(z_0)$ we have also shown that $\omega_\alpha({(0,\varepsilon)}) \subset y(C_1\cap B_r(\alpha)) \cup y_k(C_2\cap B_r(\alpha + \varepsilon \unit_1))$ where $C_1$ and $C_2$ are narrow cones with tips in $\alpha$ and $\alpha + \varepsilon \unit_1$ respectively, opened towards $\alpha+\frac{\varepsilon}{2} \unit_1$ with opening angle less than $\arccos\left(\frac{\varepsilon}{\varepsilon + 3\delta}\right)$; cf.\@ also Figure \ref{figure-lens}. It is easy to see that therefore $\omega_\alpha((0,\varepsilon))$ does neither intersect $y(\Omega_{outer})$, $y_k(\Omega_{inner})$ nor $\omega_\beta$ on any of the other intervals of $G$.
  \item On each of the intervals $[0,\varepsilon/2]$ as well as $[\varepsilon/2, \varepsilon]$ the function $\omega_\alpha$ is a composition of the $\eta$-quasisymmetric functions $y$ and $y_k$ with the $b$-bi-Lipschitz functions $\phi_1$ and $\phi_2$, respectively. This yields the claim.
  \end{enumerate}
  
\end{proof}

\begin{figure}
\begin{center}
\includegraphics[width = 0.3 \paperwidth]{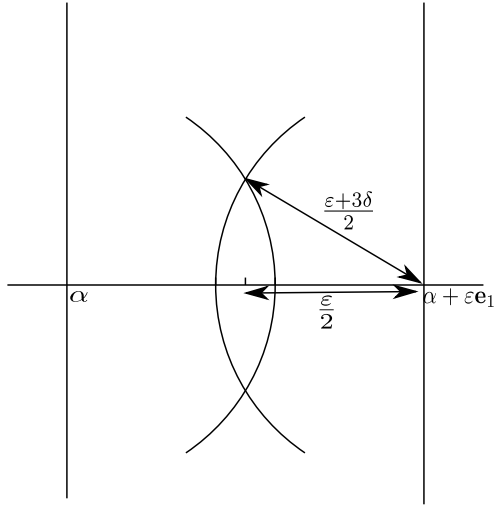}
\caption{Illustration for the proof of Lemma \ref{lemma-constructionProperties}-item 2}
\label{figure-lens}
\end{center}
\end{figure}

\begin{lemma} \label{lemma-quasicircle}
 Let $\delta$ satisfy \eqref{delta_choice} as well as $\delta < \frac{\varepsilon}{5}$. Let $y, y_k$ be as in Proposition \ref{prop-cutOff} and fulfill \eqref{needs-CutOff}. Then the image of $\omega_\alpha$ is a segment of a $\tilde{K}$-quasicircle where $\tilde{K}$ depends only on $K$.
\end{lemma}

\begin{proof}
  To show this, we will verify the classical arc condition \eqref{Arc-condition}. Take $z_1, z_2, z_3 \in \omega_\alpha((0,\varepsilon))$ and denote $x_i = \omega_\alpha^{-1}(z_i)$. We can immediately assume that $x_1 < x_2 < x_3$. If $\varepsilon/2 < x_1$ or $x_3 < \varepsilon/2$, the result is trivial since we are in the image of $y$ or $y_k$ respectively.
  
   Now, we consider the case when $x_1 < x_2 < \varepsilon/2 < x_3$. We look at the connecting line between $z_1$ and $z_3$ and on this line we find the points $\xi_1$ and $\xi_3$ satisfying 
   $$\abs{y^{-1}(z_1)-y^{-1}(\xi_1)} = \abs{\phi_1({\alpha + x_1\unit_1})-\phi_1({\alpha + \varepsilon/2 \unit_1})} $$
   and$$\abs{y_k^{-1}(z_3)-y_k^{-1}(\xi_3)} = \abs{\phi_2({\alpha + x_3 \unit_1})-\phi_2({\alpha+\varepsilon/2\unit_1})},$$ respectively.
   
 Then, directly from Construction \ref{const-inner}, we conclude that 
 \[y(B_{\abs{\phi_1({\alpha + x_1\unit_1})-\phi_1({\alpha + \varepsilon/2 \unit_1})}}(\phi_1(x_1))) \cap y_k (B_{\abs{\phi_2({\alpha + x_3\unit_1})-\phi_2({\alpha + \varepsilon/2 \unit_1})}}(\phi_2(x_3))) = \emptyset\]
 and hence $\abs{z_1-z_3} \geq \abs{z_1-\xi_1} + \abs{z_3-\xi_3}$. 
 
 Furthermore, since $y$ and $y_k$ are locally quasisymmetric, we have 
$$
\abs{z_i-\xi_i} {\geq} \eta(1)^{-1} \abs{z_i-\omega_\alpha(\varepsilon/2)} \qquad i=1,3.
$$
Also, since $\omega_\alpha(\varepsilon/2)=y(\phi_1({\alpha + \varepsilon/2 \unit_1}))$, $z_2 = y(\phi_1({\alpha + x_2\unit_1}))$ and $z_1=y(\phi_1({\alpha + x_1\unit_1}))$ lie on image of the quasisymmetric function ${ x\mapsto }y \circ \phi_1{(\alpha+x \unit_1)}$ of the interval $[0,\varepsilon/2]$, we get directly from the quasisymmetry property that (here we also use that the bi-Lipschitz constant of $\phi_1$ is bounded by 5)
{\begin{align*} 
& \frac{\abs{z_2-\omega_\alpha(\varepsilon/2)} + \abs{z_2-z_1}}{|z_1 - \omega_\alpha(\varepsilon/2)|} \\
=& \eta\left(\frac{\abs{\phi_1(\alpha+x_2\unit_1)-\phi_1(\alpha+\varepsilon/2\unit_1)}}{\abs{\phi_1(\alpha+x_1\unit_1)-\phi_1(\alpha+\varepsilon/2\unit_1)}} \right) + \eta\left(\frac{\abs{\phi_1(\alpha+x_2\unit_1)-\phi_1(\alpha+x_1\unit_1)}}{\abs{\phi_1(\alpha+x_1\unit_1)-\phi_1(\alpha+\varepsilon/2\unit_1)}} \right) \\ 
\leq& \eta\left(25\frac{\abs{x_2-\varepsilon/2}}{\abs{x_1-\varepsilon/2}} \right) + \eta\left(25\frac{\abs{x_2-x_1}}{\abs{x_1-\varepsilon/2}} \right)\\
\leq& 2 \eta(25)
\end{align*} }
Summing all up,
  {\begin{align*}
    \abs{z_1-z_3} &\geq \eta^{-1}(1) \left(\abs{z_1-\omega_\alpha(\varepsilon/2)} +  \abs{z_3-\omega_\alpha(\varepsilon/2) }\right) \\
    &\geq (2\eta(1) \eta(25))^{-1} \left( \abs{z_1-z_2} + \abs{z_2-\omega_\alpha(\varepsilon/2)} + \abs{z_3-\omega_\alpha(\varepsilon/2)} \right) \\
    &\geq  (2 \eta(1) \eta(25))^{-1} \left(\abs{z_1-z_2} + \abs{z_2-z_3} \right),
  \end{align*}
where we used the fact that $(2\eta(25))^{-1} \leq 1$ (which is true since $\eta$ is increasing and $\eta(1) \geq \frac{\abs{z_1-z_2}}{\abs{z_1-z_2}} = 1$) to get the constant in front of all terms in the second line and the triangle inequality in the last line.}
  
  The final case $x_1 < \varepsilon/2 < x-_2 < x_3$ follows from symmetry.
\end{proof}

\begin{remark}
By arguments similar to those used in the previous proof, we can show that if the functions $y$, $y_k$ were bi-H\"{o}lder continuous, so will be $\omega_\alpha$, even with the same exponent. In particular, this applies to the case studied here (recall that quasiconformal maps are locally H\"older continuous (cf. \cite[Corollary 3.10.3]{AstalaIwaniec}).
\end{remark}

We know from Lemma \ref{lemma-quasicircle}  that there exists a quasisymmetric parameterization of the image of $\omega_\alpha$. However, what we actually need is not just some parametrization, but a parametrization that is still quasisymmetric when connected to the image of $\Omega_{outer}$ under $y$ as well as $\Omega_{inner}$ under $y_k$. 

First, we realize why the function $\omega_\alpha$ itself (which does connect in a quasisymmetric way to the images of $\Omega_\mathrm{inner}$ and $\Omega_\mathrm{outer}$) does not need to be quasisymmetric across the  ``meeting point'' $\varepsilon/2$. As noted, $\omega_\alpha$ is at least bi-Hölder-continuous, which shows that it cannot form a too sharp angle at $\varepsilon/2$. Yet, this does not exclude the possibility that both parts of $\omega_\alpha$ approach the meeting point  with ``different speeds''. To be more precise, while the bi-Hölder property shows that $\omega_\alpha(\varepsilon/2+t) - \omega_\alpha(\varepsilon/2)$ and $\omega_\alpha(\varepsilon/2) - \omega_\alpha(\varepsilon/2-t)$ are roughly co-linear for small $t$, we have no bounds on the quotient
\[\frac{\abs{\omega_\alpha(\varepsilon/2+t) - \omega_\alpha(\varepsilon/2)}}{\abs{\omega_\alpha(\varepsilon/2) - \omega_\alpha(\varepsilon/2-t)}}\]

To fix this issue, we will perform yet another (slight) modification of the construction: We will re-parametrize $\omega_\alpha$  around the meeting point $\varepsilon/2$ but keep the original parametrization close to $0$ and $\varepsilon$ in order not to run into the same kind of problems when transitioning to $y$ and $y_k$ at the endpoints of the interval. This requires a slow passage from one parametrisation to another without endangering the quasi-symmetry. Nevertheless, finding such a re-parametrization is a one-dimensional problem, which we are able to solve explicitly:

\begin{lemma}\label{lemma-1d}
Let $s:[0,a]\to [0,b]$ be an increasing, $\eta$-quasisymmetric homeomorphism. Then there exists an homeomorphism $\tilde{s}:[0,a] \mapsto [0,\ell]$, $\ell < {(3/2)}b$ such that $\tilde{s}|_{[0,a/4]} = s|_{[0,a/4]}$, $\tilde{s}'|_{[3/4a,a]} = b/a$ and $\tilde{s}$ is $\tilde{\eta}$-quasisymmetric, where $\tilde{\eta}$ is only dependent on $\eta$.
\end{lemma}

We postpone the proof of this lemma to the end of the section, since it is quite technical and we rather directly state a refinement, which gives the desired passage between parametrizations:

\begin{proposition}\label{prop-reparam}
Let $r,s:[0,a]\mapsto [0,b]$ be {increasing}, $\eta$-quasisymmetric homeomorphisms. Then there exists a number $c \in (0,a/4)$ which is only dependent on $\eta$ and a homeomorphism $\tilde{s}:[0,a] \mapsto [0,b]$, such that $\tilde{s}|_{[0,c]} = r|_{[0,c]}$, $\tilde{s}|_{[a-c,a]} = s|_{[a-c,a]}$ and $\tilde{s}$ is $\tilde{\eta}$-quasisymmetric, where $\tilde{\eta}$ is only dependent on $\eta$.
\end{proposition}

\begin{proof}
  Without loss of generality, we may consider only the case when $a = b = 1$ as the general case can be handled by following in verbatim the beginning of the proof of Lemma \ref{lemma-1d}. 
  
  First of all, we find the largest number $d \in (0,1/4]$ such that $\eta(d) \leq 1/4$. Then $r$ maps the interval $[0,d]$ into $[0,1/4]$ and $s$ maps $[1-d,1]$ into $[3/4,1]$ since
  \begin{align}\label{boundR}0 < {\eta(d^{-1})^{-1}} \leq \frac{r(d)-r(0)}{r(1)-r(0)} = r(d) \leq \eta(d) \leq 1/4\end{align}
  and in the same way 
  \begin{align}\label{boundS}{\eta(d^{-1})^{-1}} \leq 1-s({1-d}) \leq \eta(d) \leq 1/4. \end{align}

 Now, we use Lemma \ref{lemma-1d} twice to construct the pieces $\tilde{s}|_{[0,d]}$ and $\tilde{s}|_{[1-d,1]}$ in such a way that $\tilde{s}|_{[0,d/4]} = r|_{[0,d/4]}$ and $\tilde{s}|_{[1-d/4,1]} = s|_{[1-d/4,1]}$, as well as $\tilde{s}'|_{[3d/4,d]} = \frac{r(d)}{d}$ and $\tilde{s}'|_{[1-d,1-3d/4]} = \frac{1-s(1-d)}{d}$. Then we know from Lemma \ref{lemma-1d} that $\tilde{s}(d) <  3/8$ and $\tilde{s}(1-d) >  5/8$, so we can simply connect both parts of the function affinely without losing injectivity. 
  
  The resulting function $\tilde{s}$ is quasisymmetric on $[0,d]$ and $[1-d,1]$ owing to Lemma \ref{lemma-1d} and moreover, on the interval $[3d/4,1-3d/4]$ it consists of 3 affine segments. Therefore, on the interval $[3d/4,1-3d/4]$ the quasisymmetry modulus of $\tilde{s}$ is determined by the changes of the slope at the points $d$ and $1-d$, but this can easily be seen as bounded by $\eta$ using (\ref{boundR}) and (\ref{boundS}). Therefore, since the three intervals overlap, $\tilde{s}$ is $\tilde{\eta}$-quasisymmetric with $\tilde{\eta}$ only depending on $\eta$.
\end{proof}

We are now in the position to present our final construction of $\omega$:

\begin{construction}
\label{const-final}
  Take some edge $[\alpha, \alpha+\varepsilon\unit_1]$ in $G$ and the corresponding $\omega_\alpha$ from Construction \ref{const-inner}. 
  
  Using Lemma \ref{lemma-quasicircle} find a re-parametrisation $s:[0,\varepsilon]\rightarrow[0,\varepsilon]$ of $\omega_\alpha$ such that $\omega_\alpha \circ s$ is quasisymmetric. Further find the number $\mathfrak{m} \in (0, \varepsilon)$ such that $s(\mathfrak{m}) = \varepsilon/2$. 
  
 Realize that  $s = \omega_\alpha^{-1} \circ (\omega_\alpha \circ s)$ is quasisymmetric on $[0,\mathfrak{m}]$ and $[\mathfrak{m},\varepsilon]$ because $\omega_\alpha$ was constructed in such a way that it is quasisymmetric on $[0,\varepsilon/2]$ as well as $[\varepsilon/2,\varepsilon]$ (cf.\@ Lemma \ref{lemma-constructionProperties}). Hence, apply Proposition \ref{prop-reparam} to each of those intervals\ to construct a function $\tilde{s}:[0,\varepsilon] \to [0,\varepsilon]$ {that is quasisymmetric when restricted to each of them and satisfies}\footnote{For the first interval, we obtain this by applying Proposition \ref{prop-reparam} to the piecewise affine function $r$ satisfying $r(0)=0,r(\mathfrak{m}/2)=\mathfrak{m}/2,r(\mathfrak{m})=\varepsilon/2$ as well as $s$ constructed above. On the second interval we connect $s$ to the piecewise affine $r$ that fulfills $r(\mathfrak{m})=\varepsilon/2, r((\varepsilon+\mathfrak{m})/2)=(\varepsilon+\mathfrak{m})/2, r(\varepsilon) = r(\varepsilon)$.}
  \[\tilde{s}(t) = \begin{cases} t &\mbox{ for } 0 \leq t \leq \lambda \varepsilon \\s(t) &\mbox{ for } \mathfrak{m}-\lambda \varepsilon \leq t \leq \mathfrak{m}+\lambda \varepsilon \\ t &\mbox{ for } \varepsilon-\lambda \varepsilon \leq t \leq \varepsilon \end{cases}\]
  for a suitable $\lambda > 0$ determined by the number $c$ in Proposition \ref{prop-reparam} and depending only on the quasisymmetry modulus $\eta$.
  
  Now define 
  \begin{align*}
  \tilde{\omega}_\alpha(t) &= \omega_\alpha(\tilde{s}(t)) \qquad \text{ and } \\   \omega(\alpha + \unit_1 s) &= \tilde{\omega}_\alpha(s) \qquad \text{on $[\alpha,\alpha+\unit_1 \varepsilon]$}.
  \end{align*}
\end{construction}

We immediately have the following property of $\tilde{\omega}_\alpha(s)$:
\begin{lemma}\label{lemma-gQuasisym}
  Let $\delta$ satisfy \eqref{delta_choice} as well as $\delta < \frac{\varepsilon}{5}$. Let $y, y_k$ be as in Proposition \ref{prop-cutOff} and fulfill \eqref{needs-CutOff}. Then $\tilde{\omega}_\alpha$ found in Construction \ref{const-final} is well defined and $\bar{\eta}$-quasisymmetric, where $\bar{\eta}$ depends only on $K$.
\end{lemma}

\begin{proof}
We know from Lemma \ref{lemma-constructionProperties} that $\omega_\alpha$ is quasisymmetric on $[0,\varepsilon/2]$ as well as $[\varepsilon/2,\varepsilon]$ with a quasisymmetry modulus depending only on $K$. Therefore, also $s$ and hence $\tilde{s}$ are quasisymmetric on $[0,\mathfrak{m}]$ and $[\mathfrak{m}, \varepsilon]$ with a quasisymmetry modulus depending only on $K$ and so is $\tilde{\omega}_\alpha$ as a composition. 

Furthermore, we know that $\tilde{\omega}_\alpha = \omega_\alpha \circ s$ is quasisymmetric on $[\mathfrak{m}-\lambda\varepsilon,\mathfrak{m}+\lambda\varepsilon]$ per construction. But since those three intervals overlap, $\tilde{\omega}_\alpha$ is quasisymmetric on all of $[0,\varepsilon]$ and the modulus depends only on the moduli of the functions involved, which all derive from $K$.
\end{proof}

With all the ingredients at hand, we summarize the proof of Proposition \ref{prop-cutOff}:

\begin{proof}[Proof of Proposition \ref{prop-cutOff}]
We pick some $\mathrm{diam}(\Omega) >> \varepsilon > 0$ and find an appropriate $\delta$ satisfying simultaneously \eqref{delta_choice} as well as $\delta<\varepsilon/5$. We perform the partition from Construction \ref{const-outer} and define $\omega$ on $\Omega_{outer} \cup \Omega_{inner}$ by \eqref{Omega-def1} while on the grid $G$, we proceed according to Construction \ref{const-final}.

We know, due to Remark \ref{remark-quasisym}, that both $y$ and $y_k$ are $\eta$-quasisymmetric on a neighborhood of each of the squares in $\Omega_{mid}$ with $\eta$ depending only on $K$ due to Lemma \ref{lemma-locQuasi}. Therefore, by employing also Lemma \ref{lemma-gQuasisym}, $\omega$ is quasisymmetric on the boundary of every square $S_{ij} \subset \Omega_{mid}$. So, we may use the Beurling-Ahlfors extension from Lemma \ref{lemma-extension} to extend $\omega$ to a quasiconformal homeomorphism on each square of $\Omega_{mid}$ which makes $\omega$ a homeomorphism defined on all of $\Omega$ satisfying $\abs{\nabla \omega}^2 \leq \kappa(K) \det(\nabla \omega)$ for some $\kappa(K)$ depending only on $K$. Moreover, since $\omega$ coincides with $y$ in a neighborhood of $\partial \Omega$, it fulfills \eqref{Ciarlet-Necas} which makes it globally injective. In other words, $\omega$ is $\kappa(K)$-quasiconformal with $\kappa(K)$ depending only on $K$.

Finally, since  the image of every square $S \subset\Omega_{mid}$ under $\omega$ is contained in the union of the image of the given square and its neighbors under $y$ and $y_k$\footnote{To see this, consider the image of $\partial S$ under $\omega$. For each $x\in\partial S$, $\omega(x)$ is given by $y(\tilde{x})$ or $y_k(\tilde{x})$ with $\tilde{x}$ being either $x$ on $\partial \Omega_{mid}$ or given by $\tilde{x} = \phi_i(x)$, so that $\abs{\tilde{x}-x}< \varepsilon$ by Lemma \ref{lemma-constructionProperties} (we can safely ignore the reparametrisation in construction \ref{const-final} since it does not change the image). Then the image of $\partial S$ under $\omega$ is contained in the union of the images of $S$ and its neighbours under $y$ and $y_k$ and thus the same holds for $S$ since $\partial\omega(S) = \omega(\partial S)$.}, we get that (denote by $x_i$ the midpoint of such square)
\[\omega(S) \subset y(B_{3\varepsilon}(x_i)) \cup y_k(B_{3\varepsilon}(x_i)) \]
so that
\[\norm{\omega-y}{L^\infty(S_i; \mathbb{R}^2)} < 3\varepsilon + \delta.\]
Similarly, we get that
\[\norm{\omega^{-1}-y^{-1}}{L^\infty(\omega(S_i); \R^2)} < 3\varepsilon + \delta,\]
which, together with $\omega = y$ on $\Omega_{outer}$ and the fact that \ref{needs-CutOff} guarantees the corresponding $L^\infty$ bounds on $\Omega_{inner}$, yields item 1.
\end{proof}

%
%
%

To end this section, we present the proof of Lemma \ref{lemma-1d}:

\begin{proof}(of Lemma \ref{lemma-1d}--1-dimensional fitting)
First, we realize that it suffices to consider quasisymmetric homeomorphisms $s:[0,1] \to [0,1]$ and to then construct $\tilde{s}$ with $\tilde{s}_{\mid_{[0,1/4]}} = s_{\mid_{[0,1/4]}}$ and $\tilde{s}'_{\mid_{[3/4,1]}} = 1$. Indeed, in the general case compose with similarities without changing $\eta$ by defining $\hat{s}: t \mapsto b^{-1} s(a t): [0,1] \to [0,1]$. Then, if we can construct $\tilde{s}$ as specified above the function $b \tilde{s}\left(\frac{t}{a}\right)$ will have all the desired properties. Notice also that such a rescaling is equivalent to composing with similarities in the domain and in the image and thus does not change the quasi-symmetry modulus.

Moreover, we may restrict our attention to $s:[0,1] \to [0,1]$ that are additionally smooth. While in fact $s$ may not even be absolutely continuous (cf.\@ e.g.\@ \cite[Thm.\@ 3]{BeurlingAhlfors} for the construction of a counterexample), in one dimension it can be uniformly approximated by a sequence of smooth $\eta$-quasisymmetric homeomorphisms (cf.\@ \cite[Thm.\@ 7]{kelingos1966}\footnote{The smoothing used in \cite{kelingos1966} may perturb the end points of the approximation slightly, so we only have $s_k(0) \rightarrow 0$ and $s_k(1)\rightarrow 1$. Note however that affinely rescaling the image back to $[0,1]$ has no impact on the convergence, so we can assume $s_k(0)=0$ and $s_k(1)=1$.}). So if $s$ is not smooth, we can approximate it by a sequence of smooth $s_k$ and construct the corresponding functions $\tilde{s}_k$ with $\tilde{s}_k\mid_{[0,1/4]} = s_k\mid_{[0,1/4]}$ and $\tilde{s}_k'\mid_{[3/4,1]} = 1$. Furthermore we know that normalized families of quasisymmetric functions are normal and therefore $\tilde{s}_k \rightarrow \tilde{s}$ for a subsequence (cf.\@ \cite[Thm.\@ 8]{kelingos1966} or \cite[Cor.\@ 3.9.3]{AstalaIwaniec}). But then we trivially have $\tilde{s}\mid_{[0,1/4]} = s\mid_{[0,1/4]}$ and $\tilde{s}'\mid_{[3/4,1]} = 1$.

Now to start with the actual proof, let us consider the following partition of unity
  \[\psi_0(t) = \begin{cases} 1& \mbox{ for }0 \leq t \leq 1/4 \\ \frac{e^{-\frac{1}{3/4-t}}}{e^{-\frac{1}{3/4-t}}+e^{-\frac{1}{t-1/4}} } &\mbox{ for } 1/4< t < 3/4 \\ 0& \mbox{ for }3/4 \leq t \leq 1  \end{cases} \]
  and
  \[\psi_1(t) = \begin{cases} 0& \mbox{ for }0 \leq t \leq 1/4 \\ \frac{e^{-\frac{1}{t-1/4}}}{e^{-\frac{1}{3/4-t}}+e^{-\frac{1}{t-1/4}} } &\mbox{ for } 1/4< t < 3/4 \\ 1& \mbox{ for }3/4 \leq t \leq 1  \end{cases} \]
  and define $\tilde{s}$ via the integral of a convex combination of the derivatives of $s$ (recall that we assume that $s$ is smooth) and the identity; i.e.  
  \[\tilde{s}(t) := \int_0^t \psi_0(x)s'(x)+\psi_1(x) \dd x.\]
Then $\tilde{s}$ is clearly an absolutely continuous, strictly monotone homeomorphism with $\tilde{s}(0) = 0$ and 
  \[\tilde{s}(1) < \int_0^1 s'(x)+\psi_1(x) \dd x = 1 + 1/2.\]
  
  We now need to show the quasisymmetry of $\tilde{s}$. For this we verify the well-known $M$-condition \cite{Ahlfors}, i.e., for all $t\in [0,1]$ and $h > 0$ we have that 
$$
  \frac{1}{M} \leq \frac{\tilde{s}(t+h)-\tilde{s}(t)}{\tilde{s}(t)-\tilde{s}(t-h)} \leq M,
$$
which reduces to showing that for all $t\in [0,1]$ and $h \neq 0$ there exists a constant $M$ that is dependent only on $\eta$ such that 
\begin{equation}
\frac{\abs{\tilde{s}(t+h)-\tilde{s}(t)}}{\abs{\tilde{s}(t)-\tilde{s}(t-h)}} = \frac{\abs{\int_t^{t+h} \psi_0(x)s'(x)+\psi_1(x) \dd x}}{\abs{\int_{t-h}^t \psi_0(x)s'(x)+\psi_1(x) \dd x}} \leq M.
\label{Mcondition}
\end{equation}
In fact, since $\tilde{s}: [0,1] \to \R$, it suffices to verify \eqref{Mcondition} for $\abs{h} < 1/8$ because for larger $h$ we may proceed by iteration. 

Let us first verify \eqref{Mcondition} for $h > 0$ and $t > 1/2$. In this case, we know that $\psi_1(t-h) > \psi_1(3/8) > 0$ and so
\begin{align*}
\frac{\abs{\int_t^{t+h} \psi_0(x)s'(x)+\psi_1(x) \dd x}}{\abs{\int_{t-h}^t \psi_0(x)s'(x)+\psi_1(x) \dd x}} &\leq \frac{\abs{\int_t^{t+h} \psi_0(t)s'(x) \dd x}}{\abs{\int_{t-h}^t \psi_0(t)s'(x)+\psi_1(t-h) \dd x}} +\frac{\abs{\int_t^{t+h}\psi_1(t+h) \dd x}}{\abs{\int_{t-h}^t \psi_0(t)s'(x)+\psi_1(t-h) \dd x}} \\ & \leq 
\frac{\abs{\psi_0(t) \int_t^{t+h} s'(x) \dd x}}{\abs{\psi_0(t) \int_{t-h}^t s'(x) \dd x}} +\frac{\abs{\int_t^{t+h}\psi_1(t+h) \dd x}}{\abs{\int_{t-h}^t\psi_1(t-h) \dd x}} 
    \\&= \frac{\abs{ s(t+h)-s(t)}}{\abs{ s(t)-s(t-h)}} +\frac{\abs{h\psi_1(t+h)}}{\abs{h\psi_1(t-h)}} < \eta(1) + \frac{\psi_1(1)}{\psi_1(3/8)}.
\end{align*}
{Note that, strictly, the above calculation holds only for $t < 3/4$ for which we have $\phi_0(t) \neq 0$. In the other case, i.e. $t \geq 3/4$, it holds that $\phi_0(x) = 0$ for all $x\geq t$ and so the term $\frac{\abs{\int_t^{t+h} \psi_0(t)s'(x) \dd x}}{\abs{\int_{t-h}^t \psi_0(t)s'(x)+\psi_1(t-h) \dd x}}$ becomes zero and we may repeat the calculation with this term dropped.}

Similar arguments apply in the case when $t \leq 1/2$ and $h<0$ since then $\psi_0(t-h) \geq \psi_0(5/8) > 0$ and we estimate
\begin{align*}
\frac{\abs{\int_t^{t+h} \psi_0(x)s'(x)+\psi_1(x) \dd x}}{\abs{\int_{t-h}^t \psi_0(x)s'(x)+\psi_1(x) \dd x}} & \leq 
\frac{\abs{\psi_0(t+h) \int_t^{t+h} s'(x) \dd x}}{\abs{\psi_0(t-h) \int_{t-h}^t s'(x) \dd x}} +\frac{\abs{\int_t^{{t+h}}\psi_1({t}) \dd x}}{\abs{\int_{t-h}^t\psi_1(t) \dd x}} 
     < \frac{\psi_0(1)}{\psi_0(5/8)} \eta(1) + 1
\end{align*}
{ where similarly as above the term $\frac{\abs{\int_t^{{t+h}}\psi_1({t}) \dd x}}{\abs{\int_{t-h}^t\psi_1(t) \dd x}} $ does not occur for $t<1/4$.}

Now, we handle the complementary cases beginning with $t \leq 1/2, h > 0$ which are slightly more elaborate. To simplify the notation, we will use $C$ for a generic constant that is independent of the problem parameters and may change from expression to expression. 

Relying on monotonicity of $\psi_0$ and $\psi_1$ we obtain similarly as above
\begin{align*}
\frac{\abs{\int_t^{t+h} \psi_0(x)s'(x)+\psi_1(x) \dd x}}{\abs{\int_{t-h}^t \psi_0(x)s'(x)+\psi_1(x) \dd x}} &\leq \frac{\abs{\int_t^{t+h} \psi_0(t)s'(x) \dd x}}{\abs{\int_{t-h}^t \psi_0(t)s'(x)+\psi_1(t-h) \dd x}} +\frac{\abs{\int_t^{t+h}\psi_1(t+h) \dd x}}{\abs{\int_{t-h}^t \psi_0(t)s'(x)+\psi_1(t-h) \dd x}} \\ & \leq 
\frac{\abs{\psi_0(t) \int_t^{t+h} s'(x) \dd x}}{\abs{\psi_0(t) \int_{t-h}^t s'(x) \dd x}} +  \frac{\abs{\int_t^{t+h}\psi_1(t+h) \dd x}}{\abs{\int_{t-h}^t \psi_0(t)s'(x)+\psi_1(t-h) \dd x}}
    \\&\leq\eta(1) + \frac{h\psi_1(t+h)}{\psi_0(t) \abs{s(t)-s(t-h)} + h\psi_1(t-h)},
\end{align*}
where, even though the denominator of the second term may not vanish, $\psi_1(t-h)$ can; thus we cannot proceed as in the previous cases. To estimate the second term, we limit ourselves to the situation  $t+h > 1/4$ (i.e.  $t>1/8> h$) since the term vanishes otherwise and still distinguish two cases: $t-1/4 < \sqrt{h}$ and $t-1/4 \geq \sqrt{h}$. In the latter case $\psi_1(t-h)$ is strictly positive and so
\begin{align*}
 \frac{h\psi_1(t+h)}{\psi_0(t) \abs{s(t)-s(t-h)} + h\psi_1(t-h)} &\leq \frac{\psi_1(t+h)}{ \psi_1(t-h)} \\& \leq C e^{-\frac{1}{t+h-1/4}+\frac{1}{t-h-1/4}} = C e^{\frac{-t+h+1/4+t+h-1/4 }{(t+h-1/4)(t-h-1/4)}} =C e^{\frac{2h}{(t-1/4)^2-h^2}} \leq C e^{\frac{2h}{h-h^2}},
\end{align*}
which is bounded for $h\in(0,1/8)$. When  $t-1/4 < \sqrt{h}$, we write
\begin{align*}
 \frac{h\psi_1(t+h)}{\psi_0(t) \abs{s(t)-s(t-h)} + h\psi_1(t-h)} & \leq \frac{h\psi_1(t+h)}{\psi_0(t) \abs{s(t)-s(t-h)}} \leq C \frac{he^{-\frac{1}{t+h-1/4}}}{\psi_0(1/2) \abs{s(t)-s(t-h)}} \leq C h^{1-\kappa_1} e^{-\frac{1}{\sqrt{h}+h}}
 \end{align*}
 where in the last estimate we used that (cf. \cite[Thm.\@ 5 and Thm.\@ 10]{kelingos1966}) quasisymmetric maps in one dimension are bi-Hölder continuous; i.e. for all $t_1, t_2 \in (0,1)$
 $$
8^{\kappa_1} |t_1-t_2|^{\kappa_2}\geq |s(t_1) - s(t_2)| \geq 8^{-\kappa_1} |t_1-t_2|^{\kappa_1},
 $$
 where $\kappa_1$ and $\kappa_2$ are solely dependent on $\eta(1)$. However, $\lim_{h\to0} C h^{1-\kappa_1} e^{-\frac{1}{\sqrt{h}+h}} = 0$ which shows that $C h^{1-\kappa_1} e^{-\frac{1}{\sqrt{h}+h}}$ is uniformly bounded for $h \in 
(0,1/8)$.

In the remaining case when $t > 1/2, h < 0 $ we argue similarly as above:
\begin{align*}
\frac{\abs{\int_t^{t+h} \psi_0(x)s'(x)+\psi_1(x) \dd x}}{\abs{\int_{t-h}^t \psi_0(x)s'(x)+\psi_1(x) \dd x}} &\leq \frac{\psi_0(t+h)\abs{s(t)-s(t+h)}}{\abs{\psi_0(t-h)(s(t-h)-s(t))+h\psi_1(t)}} + 1 
\end{align*}
    Now to bound the first term we can assume $t+h < 3/4$ and again distinguish two cases. Either we have $3/4-t < \sqrt{-h}$ and therefore get that
    \[\frac{\psi_0(t+h)\abs{s(t)-s(t+h)}}{\abs{\psi_0(t-h)(s(t-h)-s(t))+h\psi_1(t)}} \leq \frac{\psi_0(t+h)\abs{s(t)-s(t+h)}}{h\psi_1(t)} \leq C e^{-\frac{1}{\sqrt{-h}-h}} h^{\kappa_2-1}\]
    which is bounded, or we have $3/4-t \geq \sqrt{-h}$ and hence
    \[\frac{\psi_0(t+h)\abs{s(t)-s(t+h)}}{\abs{\psi_0(t{-}h)(s(t{-}h){-}s(t))+h\psi_1(t)}} \leq \frac{\psi_0(t+h)\abs{s(t){-}s(t+h)}}{\psi_0(t{-}h)\abs{s(t{-}h){-}s(t))}}\leq C \eta(1) e^{-\frac{1}{3/4-t-h}+\frac{1}{3/4-t+h}}
    \leq C \eta(1) e^{-\frac{2h}{-h-h^2}}\]
    which is also bounded.
\end{proof}


\begin{thebibliography}{10}

\bibitem{Ahlfors-refl}
{\sc Lars~V Ahlfors}, {\em Quasiconformal reflections}, Acta Mathematica, 109
  (1963), pp.~291--301.

\bibitem{Ahlfors}
{\sc Lars~Valerian Ahlfors}, {\em {Lectures on Quasiconformal Mappings}},
  vol.~38, American Mathematical Soc., 1966.

\bibitem{hm}
{\sc Omar {Anza Hafsa} and Jean-Philippe Mandallena}, {\em {Relaxation theorems
  in nonlinear elasticity}}, in {Annales de l'Institut Henri Poincare (C) Non
  Linear Analysis}, vol.~25, Elsevier, 2008, pp.~135--148.

\bibitem{Astala}
{\sc Kari Astala}, {\em {Area distortion of quasiconformal mappings}}, Acta
  Mathematica, 173 (1994), pp.~37--60.

\bibitem{AstalaFaraco}
{\sc Kari Astala and Daniel Faraco}, {\em {Quasiregular mappings and Young
  measures}}, Proceedings of the Royal Society of Edinburgh: Section A
  Mathematics, 132 (2002), pp.~1045--1056.

\bibitem{AstalaIwaniec}
{\sc Kari Astala, Tadeusz Iwaniec, and Gaven Martin}, {\em {Elliptic Partial
  Differential Equations and Quasiconformal Mappings in the Plane (PMS-48)}},
  Princeton University Press, 2008.

\bibitem{balder}
{\sc Erik~J Balder}, {\em {A general approach to lower semicontinuity and lower
  closure in optimal control theory}}, SIAM Journal on Control and
  Optimization, 22 (1984), pp.~570--598.

\bibitem{ball81}
{\sc John~M Ball}, {\em {Global invertibility of Sobolev functions and the
  interpenetration of matter}}, Proceedings of the Royal Society of Edinburgh:
  Section A Mathematics, 88 (1981), pp.~315--328.

\bibitem{ball3}
\leavevmode\vrule height 2pt depth -1.6pt width 23pt, {\em {A version of the
  fundamental theorem for Young measures}}, in {PDEs and Continuum Models of
  Phase Transitions}, Springer, 1989, pp.~207--215.

\bibitem{ballOPEN}
\leavevmode\vrule height 2pt depth -1.6pt width 23pt, {\em {Some open problems
  in elasticity}}, in {Geometry, Mechanics, and Dynamics}, Springer, 2002,
  pp.~3--59.

\bibitem{ball-james1}
{\sc John~M Ball and Richard~D James}, {\em {Fine phase mixtures as minimizers
  of energy}}, in {Analysis and Continuum Mechanics}, Springer, 1989,
  pp.~647--686.

\bibitem{BallMurat}
{\sc John~M Ball and Fran\c{c}ois Murat}, {\em {$W^{1,p}$-quasiconvexity and
  variational problems for multiple integrals}}, Journal of Functional
  Analysis, 58 (1984), pp.~225--253.

\bibitem{Belinskii}
{\sc P~Po Belinskii}, {\em {General Properties of Quasiconformal mappings}},
  Nauka, Novosibirsk, 1974.

\bibitem{bbmk2013}
{\sc Barbora Bene\v{s}ov{\'a} and Martin Kru\v{z}{\'i}k}, {\em
  {Characterization of gradient Young measures generated by homeomorphisms in
  the plane}}, ESAIM:Control, Optimisation and Calculus of Variations,  (2015),
  p.~printed electronically.

\bibitem{bbmkgpYm}
{\sc Barbora Bene\v{s}ov{\'a}, Martin Kru\v{z}{\'i}k, and Gabriel Path{\'o}},
  {\em {Young measures supported on invertible matrices}}, Applicable Analysis,
  93 (2014), pp.~105--123.

\bibitem{BeurlingAhlfors}
{\sc Arne Beurling and Lars Ahlfors}, {\em {The boundary correspondence under
  quasiconformal mappings}}, Acta Mathematica, 96 (1956), pp.~125--142.

\bibitem{Gehring}
{\sc B~Bojarski}, {\em {Homeomorphic solutions of Beltrami systems}}, in {Dokl.
  Akad. Nauk. SSSR}, vol.~102, 1955, pp.~661--664.

\bibitem{ciarlet}
{\sc Philippe~G Ciarlet}, {\em {Mathematical Elasticity, Volume I}},
  North-Holland, Amsterdam,  (1988).

\bibitem{ciarlet-necas}
{\sc Philippe~G Ciarlet and Jind\v{r}ich Ne\v{c}as}, {\em {Injectivity and
  self-contact in nonlinear elasticity}}, Archive for Rational Mechanics and
  Analysis, 97 (1987), pp.~171--188.

\bibitem{Conti}
{\sc Sergio Conti and Georg Dolzmann}, {\em On the theory of relaxation in
  nonlinear elasticity with constraints on the determinant}, Archive for
  Rational Mechanics and Analysis,  (2014), pp.~1--25.

\bibitem{dacorogna}
{\sc Bernard Dacorogna}, {\em {Direct Methods in the Calculus of Variations}},
  vol.~78, Springer, 2007.

\bibitem{daneri-pratelli}
{\sc Sara Daneri and Aldo Pratelli}, {\em {Smooth approximation of bi-Lipschitz
  orientation-preserving homeomorphisms}}, in {Annales de l'Institut Henri
  Poincare (C) Non Linear Analysis}, vol.~31, Elsevier, 2014, pp.~567--589.

\bibitem{fonseca-gangbo}
{\sc Irene Fonseca and Wilfrid Gangbo}, {\em {Degree Theory in Analysis and
  Applications}}, Clarendon Press Oxford, 1995.

\bibitem{fonseca-leoni}
{\sc Irene Fonseca and Giovanni Leoni}, {\em {Modern Methods in the Calculus of
  Variations: $L^p$ Spaces}}, Springer, 2007.

\bibitem{currents}
{\sc Mariano Giaquinta, Guiseppe Modica, and Jiri Sou\v{c}ek}, {\em {Cartesian
  Currents in the Calculus of Variations II: Variational Integrals}}, vol.~2,
  Springer, 1998.

\bibitem{Henao}
{\sc Duvan Henao and Carlos Mora-Corral}, {\em {Invertibility and weak
  continuity of the determinant for the modelling of cavitation and fracture in
  nonlinear elasticity}}, Archive for Rational Mechanics and Analysis, 197
  (2010), pp.~619--655.

\bibitem{Hencl}
{\sc Stanislav Hencl and Pekka Koskela}, {\em {Lectures on Mappings of Finite
  Distortion}}, Springer, 2014.

\bibitem{mora3}
{\sc Stanislav Hencl and Aldo Pratelli}, {\em Diffeomorphic approximation of
  $w^{1,1}$ planar sobolev homeomorphisms}, arXiv preprint 1502.07253,  (2015).

\bibitem{iwaniec1}
{\sc Tadeusz Iwaniec, Leonid~V Kovalev, and Jani Onninen}, {\em {Diffeomorphic
  approximation of Sobolev homeomorphisms}}, Archive for Rational Mechanics and
  Analysis, 201 (2011), pp.~1047--1067.

\bibitem{iwaniec3}
{\sc Tadeusz Iwaniec and Jani Onninen}, {\em {Quasiconformal Hyperelasticity
  when Cavitation is not allowed}}.
\newblock http://math.syr.edu/Library/iwaniec\%20pdfs/QuasiconfHyperel.pdf.

\bibitem{iwaniec2}
\leavevmode\vrule height 2pt depth -1.6pt width 23pt, {\em {Hyperelastic
  deformations of smallest total energy}}, Archive for Rational Mechanics and
  Analysis, 194 (2009), pp.~927--986.

\bibitem{kelingos1966}
{\sc J.~A. Kelingos}, {\em {Boundary correspondence under quasiconformal
  mappings.}}, The Michigan Mathematical Journal, 13 (1966), pp.~235--249.

\bibitem{k-p1}
{\sc David Kinderlehrer and Pablo Pedregal}, {\em {Characterizations of Young
  measures generated by gradients}}, Archive for Rational Mechanics and
  Analysis, 115 (1991), pp.~329--365.

\bibitem{k-p2}
\leavevmode\vrule height 2pt depth -1.6pt width 23pt, {\em {Gradient Young
  measures generated by sequences in Sobolev spaces}}, The Journal of Geometric
  Analysis, 4 (1994), pp.~59--90.

\bibitem{koskela}
{\sc Pekka Koskela and Jani Onninen}, {\em Mappings of finite distortion: decay
  of the jacobian in the plane}, Advances in Calculus of Variations, 1 (2008),
  pp.~309--321.

\bibitem{krw}
{\sc Konstantinos Koumatos, Filip Rindler, and Emil Wiedemann}, {\em
  {Orientation-preserving Young measures}}, arXiv preprint arXiv:1307.1007,
  (2013).

\bibitem{Filip2}
\leavevmode\vrule height 2pt depth -1.6pt width 23pt, {\em {Differential
  inclusions and Young measures involving prescribed Jacobian}}, SIAM Journal
  on Mathematical Analysis,  (2015), pp.~1169--1195.

\bibitem{MKTR}
{\sc Martin Kru\v{z}{\'i}k and Tom{\'a}\v{s} Roub{\'i}\v{c}ek}, {\em {Explicit
  Characterization of $L^p$-Young Measures}}, Journal of Mathematical Analysis
  and Applications, 198 (1996), pp.~830--843.

\bibitem{lehto}
{\sc O.~Lehto}, {\em Univalent Functions and Teichm{\"u}ller Spaces}, Graduate
  Texts in Mathematics, Springer New York, 2012.

\bibitem{mora2}
{\sc Carlos Mora-Corral and Stanislav Hencl}, {\em Diffeomorphic approximation
  of continuous almost everywhere injective sobolev deformations in the plane},
  preprint,  (2015).

\bibitem{mora1}
{\sc Carlos Mora-Corral and Aldo Pratelli}, {\em Approximation of piecewise
  affine homeomorphisms by diffeomorphisms}, The Journal of Geometric Analysis,
  24 (2014), pp.~1398--1424.

\bibitem{morrey}
{\sc Charles~Bradfield Morrey and Joseph~L Doob}, {\em {Multiple Integrals in
  the Calculus of Variations}}, vol.~4, Springer, 1966.

\bibitem{mueller}
{\sc Stefan M{\"u}ller}, {\em {Variational models for microstructure and phase
  transitions}}, Calculus of Variations and Geometric Evolution Problems,
  (1999), pp.~85--210.

\bibitem{pedregal}
{\sc Pablo Pedregal}, {\em {Parametrized Measures and Variational Principles}},
  Birkh{\"a}user, 1997.

\bibitem{tang}
{\sc Tang Qi}, {\em {Almost-everywhere injectivity in nonlinear elasticity}},
  Proceedings of the Royal Society of Edinburgh: Section A Mathematics, 109
  (1988), pp.~79--95.

\bibitem{reed}
{\sc Terence Reed}, {\em On the boundary correspondence of quasiconformal
  mappings of domains bounded by quasicircles}, Pacific Journal of Mathematics,
  28 (1969), pp.~653--661.

\bibitem{r}
{\sc Tomas Roub\'{i}\v{c}ek}, {\em {Relaxation in Optimization Theory and
  Variational Calculus}}, vol.~4, Walter de Gruyter, 1997.

\bibitem{tartar}
{\sc Luc Tartar}, {\em {Beyond Young measures}}, Meccanica, 30 (1995),
  pp.~505--526.

\bibitem{tartar1}
\leavevmode\vrule height 2pt depth -1.6pt width 23pt, {\em {Mathematical tools
  for studying oscillations and concentrations: from Young measures to
  H-measures and their variants}}, in {Multiscale Problems in Science and
  Technology}, Springer, 2002, pp.~1--84.

\bibitem{valadier}
{\sc Michel Valadier}, {\em {Young measures}}, in {Methods of Nonconvex
  Analysis}, Springer, 1990, pp.~152--188.

\bibitem{silhavy}
{\sc Miroslav \v{S}ilhav\'{y}}, {\em {The Mechanics and Thermodynamics of
  Continuous Media}}, Springer, 1997.

\bibitem{warga}
{\sc Jack Warga}, {\em {Optimal Control of Differential and Functional
  Equations}}, Academic Press, 1972.

\bibitem{y}
{\sc Laurence~Chisholm Young}, {\em {Generalized curves and the existence of an
  attained absolute minimum in the calculus of variations}}, Comptes Rendus de
  la Soci{\'e}t{\'e} des Sciences et des Lettres de Varsovie, 30 (1937),
  pp.~212--234.

\end{thebibliography}
\end{document}